\documentclass[a4paper,10pt]{amsart}
\usepackage[english]{babel}
\usepackage{amsmath,tikz-cd,pdftexcmds}
\usepackage{amssymb}
\usepackage{mathrsfs}
\usepackage{enumerate}
\usepackage{mathtools,yfonts}
\usepackage{tikz,tkz-euclide}
\usepackage{pgfplots}
\usepackage{hyperref}
\usetikzlibrary{arrows}
\usepackage{makecell}

\usepackage{tcolorbox}
\usepackage[]{algorithm2e}

\newtheorem{thm}{Theorem}[section]
\newtheorem{Lemma}[thm]{Lemma}
\newtheorem{Proposition}[thm]{Proposition}
\newtheorem{Corollary}[thm]{Corollary}

\newtheorem*{thm*}{Theorem}

\theoremstyle{definition}

\newtheorem{Definition}[thm]{Definition}
\newtheorem{Remark}[thm]{Remark}

\newtheorem{say}[thm]{}

\definecolor{wwwwww}{rgb}{0.4,0.4,0.4}

\newcommand{\Spec}{\operatorname{Spec}}

\newcommand{\G}{\mathbb{G}}

\DeclareMathOperator{\Res}{Res}

\DeclareMathOperator{\ch}{char}

\DeclareMathOperator{\mult}{mult}
\DeclareMathOperator{\Br}{Br}

\DeclareMathOperator{\Sing}{Sing}

\DeclareMathOperator{\Sec}{Sec}

\DeclareMathOperator{\secc}{\mathbb{S}ec}

\DeclareMathOperator{\rank}{rank}

\hypersetup{pdfpagemode=UseNone}
\hypersetup{pdfstartview=FitH}
		
\DeclareFontFamily{U}{cbgreek}{}
\DeclareFontShape{U}{cbgreek}{m}{n}{
        <-6>    grmn0500
        <6-7>   grmn0600
        <7-8>   grmn0700
        <8-9>   grmn0800
        <9-10>  grmn0900
        <10-12> grmn1000
        <12-17> grmn1200
        <17->   grmn1728
      }{}
\DeclareFontShape{U}{cbgreek}{bx}{n}{
        <-6>    grxn0500
        <6-7>   grxn0600
        <7-8>   grxn0700
        <8-9>   grxn0800
        <9-10>  grxn0900
        <10-12> grxn1000
        <12-17> grxn1200
        <17->   grxn1728
      }{}

\DeclareRobustCommand{\qoppa}{%
  \text{\usefont{U}{cbgreek}{\normalorbold}{n}\symbol{19}}%
}
\DeclareRobustCommand{\Qoppa}{%
  \text{\usefont{U}{cbgreek}{\normalorbold}{n}\symbol{21}}%
}
\makeatletter
\newcommand{\normalorbold}{%
  \ifnum\pdf@strcmp{\math@version}{bold}=\z@ bx\else m\fi
}
\makeatother
				
\begin{document}

\title{Rational points on even dimensional Fermat cubics}

\author[Alex Massarenti]{Alex Massarenti}
\address{\sc Alex Massarenti\\ Dipartimento di Matematica e Informatica, Universit\`a di Ferrara, Via Machiavelli 30, 44121 Ferrara, Italy}
\email{msslxa@unife.it}

\date{\today}

\maketitle

\begin{abstract}
We show that even dimensional Fermat cubic hypersurfaces are rational over any field of characteristic different from three by producing explicit rational parametrizations given by polynomials of low degree. As a byproduct of our rationality constructions we get estimates on the number of their rational points over a number field, and a class of quadro-cubic Cremona correspondences of even dimensional projective spaces. 
\end{abstract}

\tableofcontents

\section{Introduction}
The rationality of smooth hypersurfaces is one of the oldest and most challenging problems in algebraic geometry. Recall that an $n$-dimensional variety $X$ over a field $k$ is rational if it is birational to $\mathbb{P}^n_{k}$, $X$ is unirational if there is a dominant rational map $\mathbb{P}^n_{k}\dasharrow X$, and $X$ is stably rational if $X\times \mathbb{P}^m$ is rational for some $m\geq 0$. Hence, a rational variety is stably rational, and a stably rational variety is unirational. In purely algebraic terms, $X$ is rational over $k$ if the function field $k(X)$ of $X$ is isomorphic to the field of rational functions $k(x_1,\dots,x_n)$, and $X$ is unirational if there is a finite extension of $k(X)$ which is a purely transcendental field extension of $k$.

The first examples of stably rational non-rational varieties had been given in \cite{BCSS85}, where the authors, using Ch\^atelet surfaces, constructed a complex non-rational conic bundle $T$ such that $T\times \mathbb{P}^3$ is rational.

Several results, mostly concerning the non rationality of Fano hypersurfaces, appeared in the last decades and also in recent years \cite{Ko95}, \cite{Vo15}, \cite{CP16}, \cite{To16}, \cite{HKT16}, \cite{AO18}, \cite{BG18}, \cite{HPT18}, \cite{Sc19a}, \cite{Sc19b}, \cite{HPT19}. In \cite[Theorem 1.17]{CP16} J. L. Colliot-Th\'{e}l\`ene and A. Pirutka proved that a very general smooth complex quartic $3$-fold is not stably rational. In \cite[Corollary 1.4]{Sc19b} S. Schreieder gave the first examples of unirational non stably rational smooth hypersurfaces. 

Concerning the rationality of quadric hypersurface there is not much to say: a quadric $X^N\subset\mathbb{P}^{N+1}$ is rational if and only if it has a smooth rational point $p\in X^N$, and this can be seen by projection from $p$. Hence, the first interesting case, which turns out to be very difficult, is that of cubics. Thank to the work of B. Segre over the rationals \cite{Seg43} and of J . Koll\'ar \cite{Kol02} over an arbitrary field we know that a smooth cubic hypersurface is unirational if and only if it has a rational point. Unfortunately, our understanding of the rationality of cubics is much more clouded. Smooth cubic curves are non rational while cubic surfaces over an algebraically close field are rational being blow-ups of the projective plane, and smooth cubic $3$-folds are non rational as proved by C. Clemens and P. Griffiths \cite{CG72}. 

Now, let us briefly discuss the $4$-dimensional case. B. Hassett introduced subvarieties $\mathcal{C}_d$ inside the moduli space of smooth cubic $4$-folds $\mathcal{C}$ as loci parametrizing isomorphism classes of cubic $4$-folds $X\subset\mathbb{P}^5$ such that $H^{2,2}(X,\mathbb{Z})$ contains a sublattice $K\subset H^{2,2}(X,\mathbb{Z})$ whose discriminant is equal to $d\in\mathbb{Z}$, and proved that these loci are either empty or divisors in $\mathcal{C}$ \cite{Has99}, \cite{Has00}. We refer to Section \ref{4F} for all the needed details on these objects. Kuznetsov's conjecture predicts that a smooth cubic $4$-fold $X\subset\mathbb{P}^5$ is rational if and only if its class $[X]\in\mathcal{C}$ belongs to a divisor $\mathcal{C}_d$ such that $d> 6$ is not divisible by $4,9$ or any odd prime number congruent to $2$ modulo $3$. To be precise B. Hassett asked whether for a cubic $4$-fold to lie in a divisor $\mathcal{C}_d$ with $d$ admissible was equivalent to rationality. Later on, A. Kuznetsov conjectured that a cubic $4$-fold is rational if and only if it has an associated $K3$ surface in a suitable derived categorical sense \cite{Kuz10}. The equivalence of Hassett's and Kuznetsov's conditions has been proved by N. Addington and R. Thomas \cite[Theorem 1.1]{AT14}.

Certain divisors $\mathcal{C}_d$ admits a neat geometric description. For instance $\mathcal{C}_8$ parametrizes cubic $4$-folds containing a plane, $\mathcal{C}_{12}$ cubic $4$-folds containing a cubic scroll, $\mathcal{C}_{14}$ cubic $4$-folds containing a quintic del Pezzo surface or equivalently a quartic scroll, $\mathcal{C}_{20}$ cubic $4$-folds containing a Veronese surface. The rationality of a general cubic $4$-fold in $\mathcal{C}_{14}$ has been proved by U. Morin in \cite{Mor40} and by G. Fano in \cite{Fan43}. This fact has then been extended to any smooth cubic $4$-fold in $\mathcal{C}_{14}$ by M. Bolognesi, F. Russo and G. Staglian\`o \cite{BRS19}, and by M. Kontsevich and Y. Tschinkel who proved that rationality specializes in smooth families \cite{KT19}. Furthermore, thank to the work of F. Russo and G. Staglian\`o we know that any cubic $4$-fold in $\mathcal{C}_{26}$ and $\mathcal{C}_{38}$ is rational \cite{RS19}, and that this holds also for any cubic $4$-fold in $\mathcal{C}_{42}$ \cite{RS23}. We will be particularly interested in the divisors $\mathcal{C}_{8}$ and $\mathcal{C}_{14}$ since the Fermat cubic $4$-fold lies in their intersection.

Despite this great amount of efforts the rationality problem for cubic hypersurfaces is still widely open and many natural questions remain unanswered. For instance, the general cubic $4$-fold is expected to be non rational but we do not have a proof yet, and not a single example of a smooth odd dimensional rational cubic hypersurface is known. On the other hand, in all even dimensions there are smooth and rational cubics. The simplest example is that of cubics $X^{2n}\subset\mathbb{P}^{2n+1}$ containing two skew $n$-plane defined over the base field. By taking two general points, one on each $n$-plane, and associating to them the third intersection point of $X^{2n}$ with the line they generate one gets a rationality construction for $X^{2n}$. Even dimensional Fermat cubics contain many $n$-planes but any two of those defined over the base field intersect. However, $X^{2n}$ contains several pairs of skew $n$-planes which are defined over a quadratic extension of the base field and Galois conjugate. This observation will be the key of our rationality constructions. 

Recall that the Brauer group $\Br(X)$ of a projective variety $X$ is the torsion subgroup of the \'etale cohomology group $H^2(X,G_m)$, and the Brauer group $\Br(k)$ of a field $k$ is the abelian group of similarity classes of finite central simple $k$-algebras, where two finite central simple $k$-algebras $A,B$ are similar if the $k$-algebras of $a\times a$ matrices with entries in $A$ and of $b\times b$ matrices with entries in $B$ are isomorphic for some $a,b > 0$. For smooth projective varieties over a field, the Brauer group is a birational invariant, and the Brauer group of a projective space over a field $k$ is isomorphic to $\Br(k)$. Our main results in Sections \ref{S4F} and \ref{res_scal} can be summarized as follows:

\begin{thm}\label{main1}
Let $k$ be a field of characteristic $\ch(k)\neq 3$. For any $n\geq 1$ the Fermat cubic hypersurface
$$
X^{2n} = \{x_0^3+\dots + x_{2n+1}^3 = 0\}\subset\mathbb{P}^{2n+1}
$$
is rational over $k$, that is the function field $k(X^{2n})$ is isomorphic to the function field $k(x_1,\dots,x_{2n})$. In particular, $\Br(X^{2n})$ is isomorphic to $\Br(k)$.

More precisely, for all $n\geq 2 $ there exists a birational parametrization $\mathbb{P}^{2n}\dasharrow X^{2n}$ given by homogeneous polynomials of degree four, and for $n = 1$ there is a birational parametrization $\mathbb{P}^{2}\dasharrow X^{2}$ given by homogeneous polynomials of degree three. 
\end{thm}

Note that when $\ch(k)= 3$ we have that $X^{2n} = \{(x_0+\dots + x_{2n+1})^3 = 0\}\subset\mathbb{P}^{2n+1}$. The rationality of cubic surfaces $X^{2n}\subset\mathbb{P}^{2n+1}$ containing a pair of skew and conjugate $n$-planes was already known \cite[Remark 2.4.1(b)]{CSS87}. The main novelty of our approach consists in the construction of explicit rational parametrizations given by polynomials of low degree. A parameterization of $X^2$ with quartic polynomials was given in \cite[Chapter XIII, Section 13.7]{HW08}.

We will give two rationality constructions for $X^{2n}$ and we will make both of them very explicit in terms of linear systems. The one in Section \ref{gA} relies on Grassmannians of line and is more geometric in nature while that in Section \ref{res_scal} is algebraic and uses restriction of scalars. We will describe this second construction in detail, and we will see that the corresponding birational maps $\mathbb{P}^{2n}\dasharrow X^{2n}$ and $X^{2n}\dasharrow \mathbb{P}^{2n}$ are given by linear systems of quartics and quadrics respectively, although an extra care on the base locus of the linear system of quartics needs to be taken in characteristic two, we will deal with this in Section \ref{ch2}. Furthermore, it will be clear that both these rationality construction work more generally for even dimensional cubic hypersurfaces containing a pair of skew half dimensional Galois conjugate linear spaces.

Moreover, investigating the relation between these two rationality constructions in Section \ref{bri} we describe a new class of quadro-cubic Cremona transformations.

\begin{thm}\label{Cre}
Let $k$ be a field of characteristic zero. Consider the following subschemes
$$
\begin{array}{llll}
T_1 & = & \{t_0^2+t_0t_1+t_1^2 = t_2 = 0\}; & \\ 
T_2 & = & \{t_0 = t_1 = 0\}; & \\ 
T_3 & = & \{t_0 = t_{2i+1} = 0\} & \text{for } i = 0,\dots,n-1;\\ 
T_4 & = & \{t_0 = t_{2i}t_{2j+1}-t_{2i-1}t_{2j+2} = 0\} & \text{for } i = 1,\dots,n-1;\: i\leq j\leq n-1;
\end{array} 
$$
in $\mathbb{P}^{2n}_{(t_0,\dots,t_{2n})}$, and the following subschemes
$$
\begin{array}{llll}
U_1 & = & \{u_1 = u_2 = 0\};
\end{array} 
$$
and
$$
U_2 = \left\lbrace
\begin{array}{llll}
u_0 & = & 0; &  \\ 
u_{2i+1}^2 + 3u_{2i+2}^2 & = & 0 & \text{for } i = 0,\dots,n;\\ 
u_{2i+1}u_{2j+1} + 3u_{2i+2}u_{2j+2} & = & 0 & \text{for } i = 0,\dots,n-2;\: i < j \leq n-1;\\ 
u_{2i}u_{2j+1}-u_{2i-1}u_{2j+1} & = & 0 & \text{for } i = 1,\dots,n-1;\: i \leq j \leq n-1;
\end{array}\right.
$$
in $\mathbb{P}^{2n}_{(u_0,\dots,u_{2n})}$. If $n\geq 2$ then the linear system of quadric hypersurfaces containing $U_1,U_2$ and the linear system of cubic hypersurfaces of $\mathbb{P}^{2n}_{(t_0,\dots,t_{2n})}$ containing $T_1,T_2,T_4$ and vanishing with multiplicity two on $T_3$ yield a quadro-cubic Cremona correspondence between $\mathbb{P}^{2n}_{(u_0,\dots,u_{2n})}$ and $\mathbb{P}^{2n}_{(t_0,\dots,t_{2n})}$.
\end{thm}

Next, we focus on the $4$-dimensional case. Let $X^4\subset\mathbb{P}^5$ be a general cubic $4$-fold containing a del Pezzo surface $S\subset\mathbb{P}^5$ of degree five. Fano's rationality construction for $X^4$ amounts to take the restriction to $X^4$ of the map induced by the linear system of quadrics containing $S$. The following result is aimed to show that Fano's construction specializes to our rationality construction for the Fermat cubic $4$-fold and that such specialization can be carried out entirely over the base field. 
\begin{thm}\label{th_spec}
Consider the following family of cubic $4$-folds
$$
X_t^4 = \{G_t = x_0^3 + x_1^3 + x_2^3 + x_3^3 + x_4^3 + x_5^3 + tA + t^2B =  0\}\subset \mathbb{P}^{5}_{k}
$$
parametrized by $t\in\mathbb{A}_k^1$, where $k$ has characteristic zero and
$$
A = x_1x_4^2 - x_0x_4x_5 + 3x_2x_4x_5 + x_2x_5^2 - 2x_3x_5^2, \:  B = x_0^2x_5 - x_0x_2x_5 + x_1x_2x_5 + x_2^2x_5 + x_0x_3x_5 - x_2x_3x_5.
$$
Then there exists a family $\varphi_t:\mathbb{P}^4\dasharrow X_t^4$ of rational maps such that 
\begin{itemize}
\item[(i)] $\varphi_{t}:\mathbb{P}^4\dasharrow X_{t}^4$ is a birational parametrization over $k$ of $X_{t}^4$ for $t\in k$ general;
\item[(ii)] $\varphi_{0}:\mathbb{P}^4\dasharrow X_{0}^4$ is a birational parametrization of the Fermat cubic $4$-fold.
\end{itemize}
Furthermore, let $S'_t$ be the family of surfaces given by the base loci of the $\varphi_t$, and $S_t$ the family of surfaces given by the base loci of inverses of the $\varphi_t$. Then
\begin{itemize}
\item[(iii)] $S'_{t}$ is a smooth surface of degree nine for $t\in k$ general and $S'_{0}$ is a surface of degree nine which is singular along two skew lines; 
\item[(iv)] $S_{t}$ is a smooth del Pezzo surface of degree five for $t\in k$ general and $S'_0$ is the union of two conjugate planes and a Fermat cubic surface;
\item[(v)] there exists a family of surfaces $K_t\subset\mathbb{P}^8$ such that $K_{t}$ is a smooth $K3$ surface of degree $\deg(K_{t}) = 12$ and $S'_{t}$ is the projection of $K_{t}$ from a $5$-secant $3$-plane for $t \in k$ general and the same holds also for $t = 0$.
\end{itemize}
In particular, the Fermat cubic $4$-fold $X^4_0$ lies in $\mathcal{C}_{8}\cap\mathcal{C}_{14}$.
\end{thm}

In Section \ref{sec_deg} we describe another similar family of cubic $4$-folds. These $4$-folds have the advantage of being simpler that those in Theorem \ref{th_spec} but on the other hand the surfaces in the corresponding family $K_{t}$ are more complicated. Still this family is quite interesting since it also specializes to the Fermat cubic $4$-fold and the general surface in the family $S'_t$ is a surface of degree nine and singular in four distinct points. Concerning the last statement of Theorem \ref{th_spec} we mention that by \cite[Theorem 1.2]{YY23} the Fermat cubic $4$-fold is contained in all Hassett's divisors. 

As an application of our rationality construction in Section \ref{ratpQ} we get bounds on the number of rational points of bounded height of $X^{2n}$. Let $k = \mathbb{Q}$ and $p\in\mathbb{P}^n$ a point. The reduced representative $q\in\mathbb{P}^n$ of $p$ is the point $q = [q_0:\dots:q_n]$, with $q_i\in\mathbb{Z}$, such that $p = \lambda q$ for some non zero $\lambda\in\mathbb{Q}$ and $\gcd(q_0,\dots,q_n) = 1$. The height of $p$ is defined as $$ht(p) = \max\{|q_0|,\dots,|q_n|\}.$$ 
This notion of height can be generalized to any number field, we refer to \cite[Definition 1.5.4]{BoG06} for details.

\begin{thm}\label{ratpQ}
Let $k$ be a number field, and denote by $X^{2n}_B(k)$ the set of rational points of the Fermat cubic surface $X^{2n}\subset\mathbb{P}^{2n+1}$ whose height is bounded by $B\in\mathbb{N}$:
$$
X^{2n}_B(k) = \{p\in X^{2n}(k) \: | \: ht(p)\leq B\}.
$$
Then asymptotically for $B\rightarrow\infty$ we have that 
$$
B^{\frac{2n+1}{4}}\leq \sharp X^{2n}_B(k)\leq B^{4n+2}
$$
for all $n\geq 1$. Furthermore, when $n = 2$ these bounds hold for the general cubic $4$-fold of the family $X_t^4$ in Theorem \ref{th_spec}, and when $n = 1$ the lower bound can be improved to $\sharp X^{2}_B(k)\geq B$.
\end{thm}
As for Theorem \ref{main1} it will be clear that Theorem \ref{ratpQ} holds more generally for any smooth cubic containing two skew and conjugate half dimensional linear spaces. 

In the same notation as above we recall that the anti-canonical height of a point $p\in X\subset\mathbb{P}^{2n+1}$, where $X$ is a Fano hypersurface of degree $d$, is defined as $ht_{-K}(p) = (q_0^2+\dots +q_n^2)^{\frac{2n+2-d}{2}}$. Manin's conjecture predicts the existence of a non empty Zarisky open subset $\mathcal{U}\subset X$ such that the number of points of $\mathcal{U}$ of anti-canonical height at most $B$ grows as $cB(\log B)^{\rho(X)-1}$, where $c$ is the Peyre's constant \cite{Pey95} and $\rho(X)$ is the Picard rank of $X$, for $B\rightarrow \infty$ \cite{FMT89}. As a consequence of Theorem \ref{ratpQ} we get that the number of points of anti-canonical height at most $B$ of $X^{2n}$ lies in between $B^{\frac{1}{4}}$ and $B^2$. This fits into Manin's conjecture since when $n\geq 2$ we have $\rho(X^{2n}) = 1$ and the above formula becomes $cB(\log B)^{\rho(X)-1} = cB$. In the case $n = 1$ a better lower bound for the number of points of bounded height of $X^{2}$, and more generally of del Pezzo surfaces, can be found in \cite[Theorem 1.1]{FLS18}.   

Finally, we will exhibit some applications to the number of points on finite fields of the singular complete intersections and $K3$ surfaces appearing in our rationality constructions, and discuss how these fit in the general framework initiated by P. Deligne in \cite{Del74} and extended to the singular case by C. Hooley in \cite{Hoo91}.

\subsection*{Conventions on the base field and terminology} 
Let $X$ be a variety over a field $k$. When we say that $X$ is rational, without specifying over which field, we will always mean that $X$ is rational over the base field $k$. Similarly, we will say that $X$ has a point or contains a variety with certain properties meaning that $X$ has a $k$-rational point or contains a variety defined over $k$ with the required properties. 

\subsection*{Acknowledgments}
This paper originated from a question posed by \textit{Stefan Schreieder} during the conference ``\textit{Real and complex birational geometry}'' held at the University of Milan from May $15$ to May $17$, $2023$. I want to thank Stefan and the organizers of the conference \textit{Elisabetta Colombo}, \textit{Paolo Stellari} and \textit{Luca Tasin}. I want to point out that Stefan had thought independently and more or less simultaneously about the argument in Section \ref{res_scal}. I also thank \textit{Brendan Hassett} for sharing with me insights on the (uni)rationality problem for low degree hypersurfaces, \textit{Jean-Louis Colliot-Th\'el\`ene} for pointing me out known results on the rationality of cubics, \textit{Francesco Russo} for explaining me the geometry behind Proposition \ref{seg}, \textit{Giovanni Staglian\`o} for helpful discussions on Cremona transformations, \textit{Federico Caucci} for helpful comments, and \textit{Gianluca Grassi} for carefully proofreading a preliminary version of the paper.

\section{Cubic surfaces and $4$-folds}\label{S4F}
In this section we recall the state of the art on the rationality problem for cubic $4$-folds and prove some preliminary results for the Fermat cubic surface over an arbitrary field.

\subsection{Cubic $4$-folds}\label{4F}
Smooth cubic $4$-folds are parametrized by an open subset $U$ of $\mathbb{P}^{55} = \mathbb{P}(k[x_0,\dots,x_5]_3)$. The complementary set $\Delta = \mathbb{P}^{55}\setminus U$, parametrizing singular cubic hypersurfaces, is a hypersurface in $\mathbb{P}^{55}$ of degree $\deg(\Delta) = 192$. The group $PGL(6)$ acts on $U$ simply by 
$$
\begin{array}{cccc}
 PGL(6) \times U & \rightarrow & U\\
  (\alpha,X) & \mapsto & \alpha(X)
\end{array}
$$ 
and the moduli space of smooth cubic $4$-folds is the quotient $\mathcal{C} = [U/PGL(6)]$. It is an irreducible quasi-projective variety of dimension $\dim(\mathcal{C}) = 20$.

Let $X\subset\mathbb{P}^5$ be a smooth cubic $4$-fold and consider $H^{2,2}(X,\mathbb{Z}) = H^4(X,\mathbb{Z})\cap H^2(X,\Omega_X^2)$. C. Voisin proved that if $[X]\in \mathcal{C}$ is very general then $H^{2,2}(X,\mathbb{Z}) = \mathbb{Z}[h^2]$ where $h$ is the class of a hyperplane section of $X$ \cite{Voi86}. Assume that $X$ contains a $2$-dimensional integral effective cycle $S$ such $K = \left\langle h,S\right\rangle\subset H^{2,2}(X,\mathbb{Z})$ has rank two. On the sublattice $K\subset H^{2,2}(X,\mathbb{Z})$ we have an intersection form given by the matrix
$$
\left(\begin{array}{cc}
h^4 & h^2\cdot S \\ 
S\cdot h^2 & S^2
\end{array}\right) = 
\left(\begin{array}{cc}
3 & \deg(S) \\ 
\deg(S) & S^2
\end{array}\right) 
$$
where $S^2$ is the self-intersection of $S$ in $X$. The discriminant of $K$ is defined as 
$$
|K| = 
\det\left(\begin{array}{cc}
h^4 & h^2\cdot S \\ 
S\cdot h^2 & S^2
\end{array}\right)  = 3S^2 - \deg(S)^2.
$$
B. Hassett defined the Noether-Lefschetz loci as 
$$
\mathcal{C}_d = \{[X]\in \mathcal{C} \: | \text{ there exists } K\subset H^{2,2}(X,\mathbb{Z}),\:\rank(K) = 2,\: h^2\in K,\: |K| = d\}
$$
where $d\in\mathbb{Z}$. The loci $\mathcal{C}_d$ are either empty or divisors in $\mathcal{C}$. Furthermore, $\mathcal{C}_d\neq\emptyset$ if and only if $d > 6$ and $d\equiv 0,2\mod 6$, and if $[X]\in\mathcal{C}_d$ is very general then $H^{2,2}(X,\mathbb{Z}) = \left\langle h^2,S\right\rangle$ for some algebraic surface $S\subset X$ \cite{Has99}, \cite{Has00}. When $S$ is smooth we have $S^2 = 6h^2\cdot S+3h\cdot K_S+K_S^2-\chi_S$.

For instance, if $S\subset X$ is a plane we have $S^2 = 6-9+9-3 = 3$. Then $|K| = 8$ and $\mathcal{C}_8$ is the divisor parametrizing cubic $4$-folds containing a plane. Now, let $S\subset\mathbb{P}^5$ be a degree five del Pezzo surface. Then $S^2 = 30-15+5-7 = 13$, $|K| = 14$ and hence $\mathcal{C}_{14}$ is the divisor parametrizing cubic $4$-folds containing a degree five del Pezzo surface. 

An even integer $d > 6$ is admissible if it is not divisible by $4,9$ or any odd prime number congruent to $2$ modulo $3$. B. Hassett proved that $d$ is admissible if and only if the orthogonal complement of the corresponding lattice $K$ in $H^4(X,\mathbb{Z})$ is Hodge isometric to the primitive Hodge structure $H^2(S,\mathbb{Z})_{prim}$ of a polarized $K3$ surface $S$ \cite{Has99}. Kuznetsov’s conjecture predicts that a smooth cubic $4$-fold $X\subset\mathbb{P}^5$ is rational if and only if $[X]\in\mathcal{C}_d$ with $d$ admissible.

\subsection{Cubic surfaces} 
We give three different birational parametrization of the Fermat cubic surface $X^2$ over a field of characteristic different from three. Even if at this stage it might seem they come from nowhere we will see that one of them is deeply related to the geometry of $X^2$ and can be generalized to Fermat cubic hypersurfaces of arbitrary even dimension, and that the other parametrization can be related to this special one by quadratic Cremona's transformations.  

Let $k(\xi)$ be a quadratic extension of $k$ with $\xi^2 = -3$. In $\mathbb{P}^2_{(u_0,u_1,u_2)}$ consider the points
$$
\begin{array}{ll}
p_{1,+} = [\xi:1:0], & p_{1,-} = [-\xi:1:0]; \\ 
p_{2,+} = [1+\xi:0:2], & p_{2,-} = [1-\xi:0:2]; \\ 
p_3 = [1:0:-1]; &
\end{array} 
$$
the lines $L_{1,+} = \{u_0-\xi u_1 = 0\},\: L_{1,-} = \{u_0+\xi u_1 = 0\}$, and denote by 
$$\mathcal{L}_{\overrightarrow{2p}_{1,\pm}^{L_{1,\pm}},p_{2,\pm},p_3}^{4}\subset|\mathcal{O}_{\mathbb{P}^2}(4)|$$ 
the linear system of plane quartics having multiplicity at least two in $p_{1,+}$ and $L_{1,+}$ as a fixed principal tangent at $p_{1,+}$, multiplicity at least two in $p_{1,-}$ and $L_{1,-}$ as a fixed principal tangent at $p_{1,-}$, and passing through $p_{2,+},p_{2,-},p_3$.

In $\mathbb{P}^2_{(v_0,v_1,v_2)}$ set
$$
\begin{array}{ll}
q_{1,+} = [-2\xi:\xi:1], & q_{1,-} = [2\xi:-\xi:1];\\ 
q_{2,+} = [-1+\xi:2:0], & q_{1,-} = [-1-\xi:2:0];\\ 
q_{3,+} = [0:\xi:1], & q_{3,-} = [0:-\xi:1];
\end{array} 
$$
and let
$$\mathcal{L}_{q_{1,\pm},q_{2,\pm},q_{3,\pm}}^{3}\subset|\mathcal{O}_{\mathbb{P}^2}(3)|$$
be the linear system of plane cubics passing through $q_{1,+},q_{1,-},q_{2,+},q_{1,-},q_{3,+},q_{3,+}$. 

\begin{Lemma}\label{CreEq}
The quadratic Cremona transformation
$$
\begin{array}{cccc}
 cr: & \mathbb{P}^2_{(u_0,u_1,u_2)} & \dasharrow & \mathbb{P}^2_{(v_0,v_1,v_2)}\\
  & [u_0:u_1:u_2] & \mapsto & [u_0^2 + 3u_1^2 - u_2^2:u_0u_2 + u_2^2:u_1u_2],
\end{array}
$$
induced by the the linear system of conics through $p_{1,+},p_{1,-},p_3$, yields a Cremona equivalence between the linear systems $\mathcal{L}_{\overrightarrow{2p}_{1,\pm}^{L_{1,\pm}},p_{2,\pm},p_3}^{4}$ and $\mathcal{L}_{q_{1,\pm},q_{2,\pm},q_{3,\pm}}^3$.
\end{Lemma}
\begin{proof}
The inverse of $cr$ is given by 
\stepcounter{thm}
\begin{equation}\label{crI}
\begin{array}{cccc}
 cr^{-1}: & \mathbb{P}^2_{(v_0,v_1,v_2)} & \dasharrow & \mathbb{P}^2_{(u_0,u_1,u_2)}\\
  & [v_0:v_1:v_2] & \mapsto & [v_0v_1 + v_1^2 - 3v_2^2,v_0v_2 + 2v_1v_2,v_1^2 + 3v_2^2]
\end{array}
\end{equation}
and it is induced by the linear system of conics through $q_{1,+},q_{1,-}$ and $q = [1:0:0]$. Let $C$ be a general quartic in $\mathcal{L}_{\overrightarrow{2p}_{1,\pm}^{L_{1,\pm}},p_{2,\pm},p_3}^{4}$ and $\Gamma$ its image via $cr$. Then 
$$\deg(\Gamma) = 2\deg(C)-\mult_{p_{1,+}}(C)-\mult_{p_{1,-}}(C)-\mult_{p_{3}}(C) = 3.$$
Furthermore, $cr$ contracts the line $\left\langle p_{1,+},p_3\right\rangle$ to $q_{1,+}$, the line $\left\langle p_{1,-},p_3\right\rangle$ to $q_{1,-}$ and the line $\left\langle p_{1,+},p_{1,-}\right\rangle$ to $q$. Note that $\mult_{p_{1,\pm}}(C) = 2$ and $\mult_{p_{2,\pm}}(C) =\mult_{p_{3}}(C) = 1$ yield
$$
\mult_{q_{1,\pm}}(\Gamma) = \mult_{q_{2,\pm}}(\Gamma) = 1.
$$ 
Denote with $\widetilde{\mathbb{P}}^2_{(u_0,u_1,u_2)}$ the blow-up of $\mathbb{P}^2_{(u_0,u_1,u_2)}$ at $p_{1,+},p_{1,-},p_3$ with exceptional divisors $E_{1,+},E_{1,-},E_3$, with $\widetilde{\mathbb{P}}^2_{(v_0,v_1,v_2)}$ the blow-up of $\mathbb{P}^2_{(v_0,v_1,v_2)}$ at $q_{1,+},q_{1,-},q$, and with $\widetilde{cr},\widetilde{cr}^{-1}$ the isomorphisms induced by $cr,cr^{-1}$. We summarize the situation in the following diagram:
$$
\begin{tikzcd}
{\widetilde{\mathbb{P}}^2_{(u_0,u_1,u_2)}} \arrow[rr, "\widetilde{cr}", bend left = 20] \arrow[d, "\pi_u"', shift right] &  & {\widetilde{\mathbb{P}}^2_{(v_0,v_1,v_2)}} \arrow[ll, "\widetilde{cr}^{-1}"', bend left = 20, shift right] \arrow[d, "\pi_v"] \\
{\mathbb{P}^2_{(u_0,u_1,u_2)}} \arrow[rr, "cr",dotted, bend left = 20]                                                &  & {\mathbb{P}^2_{(v_0,v_1,v_2)}} \arrow[ll, "cr^{-1}"', dotted, bend left = 20]                                               
\end{tikzcd}
$$
Since $C$ has fixed principal tangents $L_{1,+},L_{1,-}$ at $p_{1,+},p_{1,-}$ its strict transform $\widetilde{C}$ intersects $E_{1,+},E_{1,-}$ in two fixed conjugate points that are mapped to $q_{3,+},q_{3,-}$ by $\pi_v\circ\widetilde{cr}$. Hence, $\Gamma$ belongs to the linear system $\mathcal{L}_{q_{1,\pm},q_{2,\pm},q_{3,\pm}}^{3}$. Finally, arguing similarly on $cr^{-1}$, and noting that $q_{1,+},q_{3,+},q$ lie on the line $\{v_1-\xi v_2 =0\}$ and $q_{1,-},q_{3,-},q$ lie on the line $\{v_1+\xi v_2 =0\}$, we get that it maps sections of $\mathcal{L}_{q_{1,\pm},q_{2,\pm},q_{3,\pm}}^{3}$ to sections of $\mathcal{L}_{\overrightarrow{2p}_{1,\pm}^{L_{1,\pm}},p_{2,\pm},p_3}^{4}$.
\end{proof}

In $\mathbb{P}^3_{(x_0,x_1,x_2,x_3)}$ consider the lines 
$$
\begin{array}{l}
L_{+} = \{2x_0-(1+\xi)x_1 = 2x_2-(1+\xi)x_3 = 0\};\\
L_{-} = \{2x_0-(1-\xi)x_1 = 2x_2-(1-\xi)x_3 = 0\};
\end{array} 
$$
the point $s = [1:-1:0:0]$, and let $\mathcal{L}_{L_{\pm},s}^2\subset |\mathcal{O}_{\mathbb{P}^3}(2)|$ be the linear system of quadric surfaces containing $L_{+},L_{-},s$.

\begin{Proposition}\label{BiRP2}
The linear system $\mathcal{L}_{\overrightarrow{2p}_{1,\pm}^{L_{1,\pm}},p_{2,\pm},p_3}^{4}$ yields a birational parametrization of the Fermat cubic surface
$$
\begin{array}{cccc}
 \varphi :& \mathbb{P}^2 & \dasharrow & X^2\subset\mathbb{P}^3\\
  & [u_0:u_1:u_2] & \mapsto & [\varphi_0:\dots:\varphi_3]
\end{array}
$$
where
\begin{small}
$$
\begin{array}{ll}
\varphi_0 =& -u_0^4 - 6u_0^2u_1^2 - 9u_1^4 - u_0u_2^3 - 3u_1u_2^3;\\ 
\varphi_1 =& u_0^4 + 6u_0^2u_1^2 + 9u_1^4 + u_0u_2^3 - 3u_1u_2^3;\\ 
\varphi_2 =& -u_0^3u_2 + 3u_0^2u_1u_2 - 3u_0u_1^2u_2 + 9u_1^3u_2 - u_2^4;\\ 
\varphi_3 =& u_0^3u_2 + 3u_0^2u_1u_2 + 3u_0u_1^2u_2 + 9u_1^3u_2 + u_2^4.
\end{array} 
$$
\end{small}
Furthermore, the linear system $\mathcal{L}_{L_{\pm},s}^2$ induces the birational inverse of $\varphi$.
\end{Proposition}
\begin{proof}
A straightforward computation shows that the image of $\varphi$ is contained in $X^2$. Moreover, the quadrics 
$$
2x_0x_2 - x_1x_2 - x_0x_3 + 2x_1x_3,\: x_2^2 - x_2x_3 + x_3^2,\: x_1x_2 - x_0x_3
$$ 
form a basis of the space of sections of $\mathcal{L}_{L_{\pm},s}^2$, and yield the birational inverse of $\varphi$. 
\end{proof}

In $\mathbb{P}^3_{(x_0,x_1,x_2,x_3)}$ consider the points 
$$
\begin{array}{lll}
r_{1,+} = [1+\xi:2:0:0], & r_{1,-} = [1+\xi:2:0:0]; &  \\ 
r_{2,+} = [0:0:\xi:1], & r_{2,-} = [0:0:-\xi:1]; &  \\ 
r_{3} = [1:-1:0:0], & r_{4} = [2:1:-2:-1], & r_{5} = [1:1:-1:-1];
\end{array} 
$$
and let $\mathcal{L}_{r_{1,\pm},r_{2,\pm},r_3,r_4,r_5}^2\subset |\mathcal{O}_{\mathbb{P}^3}(2)|$ be the linear system of quadrics through $r_{1,\pm},r_{2,\pm},r_3,r_4,r_5$.

\begin{Proposition}\label{BiRP2bis}
The linear system $\mathcal{L}_{q_{1,\pm},q_{2,\pm},q_{3,\pm}}^{3}$ yields a birational parametrization of the Fermat cubic surface
$$
\begin{array}{cccc}
 \chi: & \mathbb{P}^2 & \dasharrow & X^2\subset\mathbb{P}^3\\
  & [v_0:v_1:v_2] & \mapsto & [\chi_0:\dots:\chi_3]
\end{array}
$$
where
\begin{small}
$$
\begin{array}{ll}
\chi_0 =& v_0^3 + 2v_0^2v_1 + 2v_0v_1^2 + v_1^3 + 3v_1^2v_2 + 6v_0v_2^2 + 3v_1v_2^2 + 9v_2^3;\\ 
\chi_1 =& -v_0^3 - 2v_0^2v_1 - 2v_0v_1^2 - v_1^3 + 3v_1^2v_2 - 6v_0v_2^2 - 3v_1v_2^2 + 9v_2^3;\\ 
\chi_2 =& v_0^2v_1 + v_0v_1^2 + v_1^3 - 3v_0^2v_2 - 6v_0v_1v_2 - 3v_1^2v_2 - 3v_0v_2^2 + 3v_1v_2^2 - 9v_2^3;\\ 
\chi_3 =& -v_0^2v_1 - v_0v_1^2 - v_1^3 - 3v_0^2v_2 - 6v_0v_1v_2 - 3v_1^2v_2 + 3v_0v_2^2 - 3v_1v_2^2 - 9v_2^3.
\end{array} 
$$
\end{small}
Furthermore, the linear system $\mathcal{L}_{r_{1,\pm},r_{2,\pm},r_3,r_4,r_5}^2$ induces the birational inverse of $\chi$.
\end{Proposition}
\begin{proof}
Since $\chi = \varphi\circ cr^{-1}$, where $cr^{-1}$ is the Cremona transformation $(\ref{crI})$, the first part of the claim follows from Lemma \ref{CreEq} and Proposition \ref{BiRP2}. For the second part is enough to note that the quadrics 
$$
x_1x_2 + x_2^2 + x_0x_3 - x_1x_3 - x_2x_3 + x_3^2,\: 2x_0x_2 - x_1x_2 + 2x_2^2 - x_0x_3 + 2x_1x_3 - 2x_2x_3 + 2x_3^2,\: x_1x_2 - x_0x_3
$$ 
form a basis of the space of sections of $\mathcal{L}_{r_{1,\pm},r_{2,\pm},r_3,r_4,r_5}^2$, and yield the birational inverse of $\chi$.
\end{proof}

\begin{Remark}
The parametrization $\chi$ of $X^2$ in Proposition \ref{BiRP2bis} induces an isomorphism between $\mathbb{P}^2$ blown-up at six points and $X^2$, and hence when the six base points are not defined over the base field $\chi$ provides a parametric description of the whole of $X(k)$. 
\end{Remark}

\begin{say}\label{tc}
Next, we develop another rationality construction liked to the existence of a twisted cubic contained in $X^2$. The image of the morphism
$$
\begin{array}{cccc}
 \gamma: & \mathbb{P}^1 & \rightarrow & \mathbb{P}^3\\
  & [a_0:a_1] & \mapsto & [\gamma_0:\dots:\gamma_3]
\end{array}
$$
where
\begin{small}
$$
\begin{array}{ll}
\gamma_0 =& a_0a_1^2;\\ 
\gamma_1 =& -3a_0^2a_1 + 2a_0a_1^2 - a_1^3;\\ 
\gamma_2 =& -3a_0^3 + 3a_0^2a_1 - 2a_0a_1^2 + a_1^3;\\ 
\gamma_3 =& 3a_0^3 - 3a_0^2a_1 + 2a_0a_1^2;
\end{array} 
$$
\end{small}
is a twisted cubic $C = \gamma(\mathbb{P}^1)$ contained in the Fermat cubic surface $X^2$. It comes as the residual intersection of $X^2$ with a general quadric of $\mathbb{P}^3$ containing the three lines $L_{+},L_{-}$ and $\{x_0+x_2 = x_1+x_3 = 0\}$.

We need to recall the notion of variety with one apparent double point. Let $Z\subset\mathbb{P}^N$ be an irreducible and non degenerate variety, $\Gamma_2(Z)\subset Z\times Z\times\G(1,N)$ the closure of the graph of the rational map $\alpha: Z\times Z \dasharrow \G(1,N)$ taking two general points to their linear span, $\pi:\Gamma_2(Z)\to\G(1,N)$ the natural projection, and $\mathcal{S}_2(Z):=\pi(\Gamma_2(Z))\subset\G(1,N)$. Note that $\mathcal{S}_2(Z)$ is irreducible of dimension $2\dim(Z)$. Consider
$$\mathcal{I} = \{(z,\Lambda) \: | \: z\in \Lambda\}\subset\mathbb{P}^N\times\G(1,N)$$
with the projections $\pi_2^Z$ and $\psi_2^Z$ onto the factors. The abstract secant variety is the irreducible variety
$$\Sec_{2}(Z):=(\psi_2^Z)^{-1}(\mathcal{S}_2(Z))\subset \mathcal{I}.$$
The secant variety is defined as
$$\secc_{2}(Z):=\pi_2^Z(\Sec_{2}(Z))\subset\mathbb{P}^N.$$

\begin{Definition}
We say that an irreducible and non degenerate variety $Z\subset\mathbb{P}^{2n+1}$ of dimension $n$ has one apparent double point if $\secc_{2}(Z) = \mathbb{P}^{2n+1}$ and $\pi_2^Z:\Sec_2(Z)\rightarrow\secc_2(Z)$ is birational.  
\end{Definition}

We recall the following well-known fact.

\begin{Proposition}\label{oadp}
Let $Z\subset\mathbb{P}^{2n+1}$ be a variety with one apparent double point and $X\subset\mathbb{P}^{2n+1}$ a smooth cubic hypersurface containing $Z$. Then $X$ is rational. 
\end{Proposition}
\begin{proof}
Let $H\subset\mathbb{P}^{2n+1}$ be a general hyperplane. A general line $L\in\mathcal{S}_2(Z)$ intersects $H$ in a point. Conversely, if $p\in H$ is a general point, since $Z$ has one apparent double point, there is a unique line $L_p\in \mathcal{S}_2(Z)$ passing through $p$. Hence $\mathcal{S}_2(Z)$ is rational. 

Now, a general line $L\in \mathcal{S}_2(Z)$ intersects $X$ in a third point $x_{L}\in X$. Fix a general point $x_L \in X$ and consider the fiber $F_{x_L}$ of $\pi_2^Z$ over $x_L$. Since $X$ is a divisor $F_{x_L}$ has dimension zero, and since $\pi_2^Z$ is birational and $\secc_{2}(Z) = \mathbb{P}^{2n+1}$ is smooth $F_{x_L}$ consists of just one point. Therefore, the rational map
$$
\begin{array}{ccc}
\mathcal{S}_2(Z) & \dasharrow & X\\
 L & \mapsto & x_L
\end{array}
$$
is birational and hence $X$ is rational.
\end{proof}

\begin{Remark}
The twisted cubic $C\subset X^2$ is a variety with one apparent double point and hence Proposition \ref{oadp} provides another proof of the rationality of $X^2$. Following the proof of Proposition \ref{oadp} we can write down explicitly a birational parametrization which turns out to be induced by a linear system of plane sextics having points of multiplicity at least three in a pair of conjugate points and in another point defined over the base field, and passing through three pairs of conjugate points.

This linear system of sextics can be transformed into a linear system of cubics, as the one in Proposition \ref{BiRP2bis}, via the quadratic Cremona transformation centered at the three triple points. 
\end{Remark}
\end{say}

When the base field $k$ has characteristic $\ch(k) = 2$ the parametrizations of $X^2$ in Propositions \ref{BiRP2}, \ref{BiRP2bis} and in Section \ref{tc} degenerate to fibrations onto a line contained in $X^2$. So the maps in Propositions \ref{BiRP2}, \ref{BiRP2bis} needs to be slightly modified in order to obtain birational parametrizations of $X^2$ in characteristic two.

\begin{Proposition}\label{Fer_Char2}
Let $X^2\subset\mathbb{P}^3$ be the Fermat cubic hypersurface over a field $k$ of characteristic $\ch(k) = 2$. The map
$$
\begin{array}{cccc}
 \alpha: & \mathbb{P}^2 & \rightarrow & X^2\subset\mathbb{P}^3\\
  & [u_0:u_1:u_2] & \mapsto & [\alpha_0:\dots:\alpha_3]
\end{array}
$$
where
\begin{small}
$$
\begin{array}{ll}
\alpha_0 = & u_0^4 + u_0^2u_1^2 + u_1^4 + u_0u_2^3 + u_1u_2^3;\\ 
\alpha_1 = & u_0^4 + u_0^2u_1^2 + u_1^4 + u_0u_2^3;\\ 
\alpha_2 = & u_0^3u_2 + u_0^2u_1u_2 + u_0u_1^2u_2 + u_2^4;\\ 
\alpha_3 = & u_0^3u_2 + u_1^3u_2 + u_2^4;
\end{array} 
$$
\end{small}
yields a birational parametrization of $X^2$. Furthermore, the quadratic Cremona transformation
$$
\begin{array}{cccc}
 cr: & \mathbb{P}^2_{(u_0,u_1,u_2)} & \rightarrow & \mathbb{P}^2_{(v_0,v_1,v_2)}\\
  & [u_0:u_1:u_2] & \mapsto & [u_0^2 + u_0u_1 + u_1^2 + u_2^2
:u_0u_2 + u_2^2:u_1u_2]
\end{array}
$$
provides the following alternate birational parametrization
$$
\begin{array}{cccc}
 \beta = \alpha\circ cr^{-1}: & \mathbb{P}^2 & \rightarrow & X^2\subset\mathbb{P}^3\\
  & [v_0:v_1:v_2] & \mapsto & [\beta_0:\dots:\beta_3]
\end{array}
$$
where
\begin{small}
$$
\begin{array}{ll}
\beta_0 = & v_0^3 + v_1^3 + v_0^2v_2 + v_2^3;\\ 
\beta_1 = & v_0^3 + v_1^3 + v_0^2v_2 + v_1^2v_2 + v_1v_2^2;\\ 
\beta_2 = & v_0^2v_1 + v_0v_1^2 + v_1^3 + v_1^2v_2 + v_0v_2^2 + v_1v_2^2;\\ 
\beta_3 = & v_0^2v_1 + v_0v_1^2 + v_1^3 + v_0^2v_2 + v_2^3.
\end{array} 
$$
\end{small}
\end{Proposition}
\begin{proof}
The birational inverse of $\alpha$ is given by 
$$
\begin{array}{cccc}
 \alpha^{-1}: & X^2 & \rightarrow & \mathbb{P}^2\\
  & [x_0:\dots :x_3] & \mapsto & [\alpha^{-1}_0:\dots : \alpha^{-1}_2]
\end{array}
$$
where
\begin{small}
$$
\begin{array}{ll}
\alpha^{-1}_0 = & x_0x_1^2 + x_1^3 + x_2^3 + x_2^2x_3 + x_2x_3^2 + x_3^3;\\ 
\alpha^{-1}_1 = & x_0^2x_1 + x_1^3 + x_2^2x_3 + x_3^3;\\ 
\alpha^{-1}_2 = & x_0^2x_2 + x_1^2x_2 + x_0^2x_3 + x_0x_1x_3. 
\end{array} 
$$
\end{small}
The rest of the statement can be verified by straightforward computations.
\end{proof}

\section{First rationality construction: Grassmannians of lines}\label{gA}

Let $V$ be a $k$-vector space of dimension $n+1$ and $G(r+1,V)$ the Grassmannian of $k$-vector subspaces of $V$. Fix a basis $e_0,\dots,e_n$ of $V$. Let $W\subset V$ be a $k$-vector subspace of dimension $r+1$, $\{w_0,\dots,w_r\}$ a basis of $W$ and write $w_i = t_{1}^ie_0+\dots + a_n^ie_n$ for $i = 0,\dots,r$. Consider the matrices
$$
A_1 =\left(
\begin{array}{ccc}
t_{1}^0 & \dots & a_r^0\\ 
\vdots & \ddots & \vdots \\ 
t_{1}^r & \dots & a_r^r
\end{array}\right),
\quad
A_2 =\left(
\begin{array}{ccc}
a_{r+1}^0 & \dots & a_n^0\\ 
\vdots & \ddots & \vdots \\ 
a_{r+1}^r & \dots & a_n^r
\end{array}\right)
\quad
\text{and}
\quad
A = \left(
A_1, A_2
\right).
$$ 
On the open subset $\mathcal{U} = \{\det(A_1) \neq 0\}\subset G(r+1,V)$ we can invert $A_1$ and consider the matrix
$$
A_1^{-1}A = \left(I_{r+1,r+1}, A_1^{-1}A_2\right)
$$
where $I_{r+1,r+1}$ is the $(r+1)\times (r+1)$ identity matrix. The matrix $A_1^{-1}A_2$ is an $(r+1)\times (n-r)$ matrix with entries in $k$, and its entries yield an isomorphism between $k^{(r+1)(n-r)}$ and $\mathcal{U}$. Therefore, $G(r+1,V)$ has dimension $(r+1)(n-r)$ and is rational over $k$.

We will denote by $\mathbb{G}(r,n)\cong G(r+1,V)$ the Grassmannian parametrizing $r$-planes, defined over $k$, of $\mathbb{P}^n = \mathbb{P}(V)$.

\begin{Lemma}\label{ginj}
Let $X$ and $Y$ be irreducible varieties over a field $k$ of characteristic zero. Assume that $X$ is rational over $k$. If there exists a generically injective rational map
$$
f:X\dasharrow Y
$$
defined over $k$ and $\dim(X) = \dim(Y)$ then $Y$ is also rational over $k$.
\end{Lemma}
\begin{proof}
Since $\dim(X) = \dim(Y)$ and $\ch(k) = 0$ there exists an open subset $\mathcal{V}\subset Y$ such that $f_{|f^{-1}(\mathcal{V})}:f^{-1}(\mathcal{V})\rightarrow\mathcal{V}$ is finite and \'etale. Since $f$ is generically injective $f_{|f^{-1}(\mathcal{V})}:f^{-1}(\mathcal{V})\rightarrow\mathcal{V}$ is an isomorphism. Finally, since $X$ is rational $f^{-1}(\mathcal{V})$ is rational and hence $\mathcal{V}$ and $Y$ are also rational.
\end{proof}

\begin{Proposition}\label{p1}
Let $X\subset\mathbb{P}^{2n+1}$ be a smooth cubic hypersurface over a field $k$ of characteristic zero. Assume that $X$ contains two skew $n$-planes which are conjugate over $k$. Then $X$ is rational over $k$.
\end{Proposition}
\begin{proof}
Let $H_1,H_2\subset X$ be two skew and conjugate $n$-planes, and $\Lambda\subset\mathbb{P}^{2n+1}$ an $(n-1)$-plane, defined over $k$, such that $\Lambda\cap (H_1\cup H_2) = \emptyset$.

A general $(n+1)$-plane $H\subset\mathbb{P}^{2n+1}$, defined over $k$, containing $\Lambda$ intersects $H_1$ in a point $p_1$ and $H_2$ in a point $p_2$. Since $H_1,H_2$ are conjugate $p_1,p_2$ are also conjugate. Hence, the line $L_H = \left\langle p_1,p_2\right\rangle$ is defined over $k$ and intersects $X$ in a third point $x_H \in X \cap L_H$ which is defined over $k$ as well.

The $(n+1)$-planes, defined over $k$, of $\mathbb{P}^{2n+1}$ containing $\Lambda$ are parametrized by the Grassmannian $\mathbb{G}(n-1,n+1) \cong \mathbb{G}(1,n+1)$. Hence, we get a rational map
$$
\begin{array}{cccc}
 \phi: & \mathbb{G}(1,n+1) & \dasharrow & X\\
  & H & \mapsto & x_H
\end{array}
$$
which is defined over $k$. Fix a general point $q \in \phi(\mathbb{G}(1,n+1))$, and assume that $q = \phi(H) = \phi(H')$. Since $H_1,H_2$ are skew we have that $L_H = L_{H'}$. Now, $\left\langle L_H,\Lambda\right\rangle\subset H_1\cap H_2$. Furthermore, since we can assume that $H\in \mathbb{G}(n-1,n+1)$ is general we have that $L_H\cap \Lambda = \emptyset$ and hence $H = H'$. Therefore, $\phi$ is generically injective.

Finally, to conclude it is enough to recall that $\mathbb{G}(1,n+1)$ has dimension $2n$ and is rational over $k$, and to apply Lemma \ref{ginj}.
\end{proof}

\begin{Corollary}\label{c1}
The Fermat cubic hypersurface 
$$
X^{2n} = \{x_0^3 +\dots +x_{2n+1}^3 = 0\}\subset\mathbb{P}^{2n+1}
$$
is rational over any field $k$ of characteristic zero for all $n\geq 1$.
\end{Corollary}
\begin{proof}
Let $k(\xi)$ be a quadratic extension of $k$ with $\xi^2 = -3$, and set 
$$
a_{+} = \frac{1+\xi}{2}\quad \text{and}\quad a_{-} = \frac{1-\xi}{2}.
$$
Then 
$$
H_{+} = \{x_i - a_{+}x_{i+1} = 0,\: \text{for} \: i \in [0,2n]\: \text{even}\}\quad \text{and}\quad H_{-} = \{x_i - a_{-}x_{i+1} = 0,\: \text{for} \: i \in [0,2n]\: \text{even}\}
$$
are a pair of skew conjugate $n$-planes in $X^{2n}$. Hence, the claim follows from Proposition \ref{p1}.
\end{proof}

\section{Second rationality construction: restriction of scalars}\label{res_scal}
Let $L$ be a finite extension of a field $k$ and 
$$
X = \Spec\left(\frac{L[x_1,\dots,x_n]}{(f_1,\dots,f_m)}\right)
$$ 
an affine variety over $L$. Fix a basis $e_1,\dots,e_s$ of $L$ over $k$, introduce new variables $y_{i,j}$ for $i = 1,\dots,n$, $j = 1,\dots,s$, and write
$$
x_i = \sum_{j=1}^s y_{i,j}e_i
$$
for all $i = 1,\dots,n$ and 
$$
f_r(x_1,\dots,x_n) = F_{r,1}e_1 + \dots + F_{r,s}e_s
$$
with $F_{r,i} \in k[y_{1,1},\dots,y_{1,s},\dots,y_{n,1},\dots,y_{n,s}]$. The affine variety over $k$
$$
\Res_{L/k}(X) = \Spec\left(\frac{k[y_{i,j}]}{(F_{r,i})}\right)
$$
is the restriction of scalars of $X$.

\begin{proof}[Proof of Proposition \ref{p1} via restriction of scalars]
Up to a change of variables we may assume that the two skew and conjugate $n$-planes are defined by
$$
H_{+} = \{x_i - a_{+}x_{i+1} = 0,\: \text{for} \: i \in [0,2n]\: \text{even}\}\quad \text{and}\quad H_{-} = \{x_i - a_{-}x_{i+1} = 0,\: \text{for} \: i \in [0,2n]\: \text{even}\}
$$
where
$$
a_{+} = \frac{1+\xi}{2}\quad \text{and}\quad a_{-} = \frac{1-\xi}{2}.
$$
and $k(\xi)$ is a quadratic extension of $k$ with $\xi^2 = -3$. Consider the affine chart $\{x_{2n+1}\neq 0\}$, and set $x_i = y_{i,1}+\xi y_{i,2}$ for $i = 0,\dots,2n$, and $x_{2n+1} = 1$. Then
$$
x_i-a_{+}x_{i+1} = \left(y_{i,1}-\frac{1}{2}y_{i+1,1}+\frac{3}{2}y_{i+1,2}\right) + \xi \left(y_{i,2}-\frac{1}{2}y_{i+1,1}-\frac{1}{2}y_{i+1,2}\right) = 0
$$ 
yields
$$
\left\lbrace\begin{array}{ll}
y_{i,1} =  & \frac{1}{2}y_{i+1,1} - \frac{3}{2}y_{i+1,2};\\ 
y_{i,2} = & \frac{1}{2}y_{i+1,1} + \frac{1}{2}y_{i+1,2};
\end{array}\right. \quad \text{for} \: i \in [0,2n-2] \: \text{even};
$$
and $x_{2n}-a_{+} = 0$ yields $y_{2n,1} = \frac{1}{2}$, $y_{2n,2} = \frac{1}{2}$. Similarly
$$
x_i-a_{-}x_{i+1} = \left(y_{i,1}-\frac{1}{2}y_{i+1,1}+\frac{3}{2}y_{i+1,2}\right) + \xi \left(y_{i,2}-\frac{1}{2}y_{i+1,1}-\frac{1}{2}y_{i+1,2}\right) = 0
$$ 
yields
$$
\left\lbrace\begin{array}{ll}
y_{i,1} =  & \frac{1}{2}y_{i+1,1} + \frac{3}{2}y_{i+1,2};\\ 
y_{i,2} = & -\frac{1}{2}y_{i+1,1} + \frac{1}{2}y_{i+1,2};
\end{array}\right. \quad \text{for} \: i \in [0,2n-2] \: \text{even};
$$
and $x_{2n}-a_{-} = 0$ yields $y_{2n,1} = \frac{1}{2}$, $y_{2n,2} = -\frac{1}{2}$. Hence, we get two points $x^{a_{+}}\in H_{+}$, $x^{a_{-}}\in H_{-}$ with coordinates
$$
x^{a_{+}}_i = \frac{1}{2}y_{i+1,1} - \frac{3}{2}y_{i+1,2} + \xi\left(\frac{1}{2}y_{i+1,1} + \frac{1}{2}y_{i+1,2}\right)
$$
for $i = 0,\dots, 2n-2$ even, 
$$
x^{a_{+}}_j = y_{j,1} + \xi y_{j,2}
$$
for $j = 1,\dots, 2n-1$ odd, and $x^{a_{+}}_{2n} = \frac{1}{2}+\frac{1}{2}\xi$;
$$
x^{a_{-}}_i = \frac{1}{2}z_{i+1,1} + \frac{3}{2}z_{i+1,2} + \xi\left(-\frac{1}{2}z_{i+1,1} + \frac{1}{2}z_{i+1,2}\right)
$$
for $i = 0,\dots, 2n-2$ even, 
$$
x^{a_{-}}_j = z_{j,1} + \xi z_{j,2}
$$
for $j = 1,\dots, 2n-1$ odd, and $x^{a_{-}}_{2n} = \frac{1}{2}-\frac{1}{2}\xi$, where we denoted by $y_{i,j}$ the coordinates of a point in $H_{+}$ and by $z_{i,j}$ the coordinates of a point in $H_{-}$. Then $x^{a_{+}}$ and $x^{a_{-}}$ are conjugate if and only if
\stepcounter{thm}
\begin{equation}\label{eqsub}
\left\lbrace\begin{array}{ll}
y_{i+1,1} - z_{i+1,1}  = & 0;\\ 
y_{i+1,2} + z_{i+1,2}  = & 0;
\end{array}\right. 
\end{equation}
for $i =0,\dots,2n-2$ even. Therefore, in the affine space $H_{+}\times H_{-} \cong \mathbb{A}^{4n}$ the pair of conjugate points form an affine subspace $H\cong\mathbb{A}^{2n}$ cut out by the equations in (\ref{eqsub}). Now, set 
$$y_{i,1} = u_i,y_{i,2} = u_{i+1}.$$
Then (\ref{eqsub}) yields $z_{i,1} = u_i,z_{i,2} = -u_{i+1}$ and a point $u = (u_1,\dots,u_{2n})\in H\cong\mathbb{A}_{(u_1,\dots,u_{2n})}^{2n}$ determines the pair of points $(x^{a_{+}}(u),x^{a_{-}}(u))\in H_{+}\times H_{-}$ with 
$$
x^{a_{+}}(u)_i = \frac{1}{2}u_{i+1} - \frac{3}{2}u_{i+2} + \xi\left(\frac{1}{2}u_{i+1} + \frac{1}{2}u_{i+2}\right)
$$
for $i = 0,\dots, 2n-2$ even, 
$$
x^{a_{+}}(u)_j = u_{j} + \xi u_{j+1}
$$
for $j = 1,\dots, 2n-1$ odd, and $x^{a_{+}}_{2n} = \frac{1}{2}+\frac{1}{2}\xi$;
$$
x^{a_{-}}(u)_i = \frac{1}{2}u_{i+1} - \frac{3}{2}u_{i+2} + \xi\left(-\frac{1}{2}u_{i+1} - \frac{1}{2}u_{i+2}\right)
$$
for $i = 0,\dots, 2n-2$ even, 
$$
x^{a_{-}}(u)_j = u_{j} - \xi u_{j+1}
$$
for $j = 1,\dots, 2n-1$ odd, and $x^{a_{-}}(u)_{2n} = \frac{1}{2}-\frac{1}{2}\xi$.

Since $x^{a_{+}}(u),x^{a_{-}}(u)$ are conjugate the line $\left\langle x^{a_{+}}(u),x^{a_{-}}(u)\right\rangle$ intersects $X$ in a third point $(X \cap \left\langle x^{a_{+}}(u),x^{a_{-}}(u)\right\rangle)\setminus \{x^{a_{+}}(u),x^{a_{-}}(u)\}$ defined over $k$.

Finally, arguing as in the last part of the proof of Proposition \ref{p1} we get that the rational map 
\stepcounter{thm}
\begin{equation}\label{parRS}
\begin{array}{cccc}
 \varphi: & \mathbb{A}^{2n} & \dasharrow & X\\
   & u & \mapsto & (X \cap \left\langle x^{a_{+}}(u),x^{a_{-}}(u)\right\rangle)\setminus \{x^{a_{+}}(u),x^{a_{-}}(u)\}.
\end{array}
\end{equation}
is birational.
\end{proof}

Next, we work out the birational parametrization $\varphi:\mathbb{A}^{2n}\dasharrow X^{2n}$ in (\ref{parRS}) when $X = X^{2n}$ is the Fermat cubic. Set 
$$L = \left\langle x^{a_{+}}(u),x^{a_{-}}(u)\right\rangle$$
and write parametrically 
$$
L_i = x^{a_{+}}(u)_i+\lambda(x^{a_{-}}(u)_i-x^{a_{+}}(u)_i)
$$
with $\lambda\in k$, $i = 0,\dots,2n$. Then
$$
L_i = \frac{u_{i+1}-3u_{i+2}+\xi(u_{i+1}+u_{i+2})}{2}-\lambda\xi(u_{i+1}+u_{i+2})
$$
for $i = 0,\dots,2n-2$ even,
$$
L_j = u_j+\xi u_{j+1}-2 \lambda\xi u_{j+1}
$$
for $j = 1,\dots,2n-1$ odd, and $L_{2n} = \frac{1+\xi}{2}-\lambda\xi$. Now, set
\stepcounter{thm}
\begin{equation}\label{AB}
\begin{array}{ll}
A = & \sum_{i=0}^{n-1}(u_{2i+1}^3+3u_{2i+1}^2u_{2i+2}+3u_{2i+1}u_{2i+2}^2+9u_{2i+2}^3)+1;\\ 
B = & \sum_{i=0}^{n-1}(u_{2i+1}^3-u_{2i+1}^2u_{2i+2}+3u_{2i+1}u_{2i+2}^2-3u_{2i+2}^3)+1.
\end{array}  
\end{equation}
Substituting $x_i = L_i$ for $i = 0,\dots,2n$ in $F = x_0^3+\dots + x_{2n}^3 + 1$ we get the polynomial
$$
F_u(\lambda) = 3\xi A \lambda^3 - \frac{9}{2}(\xi A + B)\lambda^2 + \frac{3}{2}(\xi A+3B)\lambda.
$$
Therefore, $L\subset X^{2n}$ if and only if $A = B = 0$. Now, assume that $A\neq 0$. Then the polynomial $F_u(\lambda)$ has the following three roots:
$$
\lambda_1 = 0, \lambda_2 = 1, \lambda_3 = \frac{A-\xi B}{2A}.
$$ 
Finally, substituting $\lambda = \lambda_3$ in the $L_i$ we get the following coordinates for the point $(X^{2n} \cap \left\langle x^{a_{+}}(u),x^{a_{-}}(u)\right\rangle)\setminus \{x^{a_{+}}(u),x^{a_{-}}(u)\}$:
$$
\varphi_i = \frac{(u_{i+1}-3u_{i+2})A - 3(u_{i+1} + u_{i+2})B}{2A}
$$
for $i = 0,\dots,2n-2$ even,
$$
\varphi_j = \frac{u_{j}A - 3u_{j+1}B}{A}
$$
for $j = 1,\dots,2n-1$ odd, and $\varphi_{2n} = \frac{A-3B}{2A}$.

We now extend the map $\varphi:\mathbb{A}^{2n}\dasharrow X^{2n}$ to a map $\overline{\varphi}:\mathbb{P}^{2n}\dasharrow X^{2n}\subset\mathbb{P}^{2n+1}$ by introducing a new variable $u_0$ and homogeneizing the polynomials $A,B$ in (\ref{AB}):
\stepcounter{thm}
\begin{equation}\label{ABu0}
\begin{array}{ll}
\overline{A} = & u_0^3 + \sum_{i=0}^{n-1}u_{2i+1}^3+3u_{2i+1}^2u_{2i+2}+3u_{2i+1}u_{2i+2}^2+9u_{2i+2}^3;\\ 
\overline{B} = & u_0^3 + \sum_{i=0}^{n-1}u_{2i+1}^3-u_{2i+1}^2u_{2i+2}+3u_{2i+1}u_{2i+2}^2-3u_{2i+2}^3.
\end{array}  
\end{equation}

\begin{say}\label{comp_m}
Homogeneizing the $\varphi_i$ we get  
$$
\overline{\varphi}_i = \frac{(u_{i+1}-3u_{i+2})\overline{A} - 3(u_{i+1} + u_{i+2})\overline{B}}{2}
$$
for $i = 0,\dots,2n-2$ even,
$$
\overline{\varphi}_j = u_{j}\overline{A} - 3u_{j+1}\overline{B}
$$
for $j = 1,\dots,2n-1$ odd, $\overline{\varphi}_{2n} = \frac{u_0(\overline{A}-3\overline{B})}{2}$, $\overline{\varphi}_{2n+1} = u_0\overline{A}$.
\end{say}

\begin{Lemma}\label{Irr}
If $n\geq 2$ then the complete intersection $Y^{2n-2} = \{\overline{A} = \overline{B} = 0\}\subset\mathbb{P}^{2n}$ defined by the polynomials in (\ref{ABu0}) is irreducible.  
\end{Lemma}
\begin{proof}
Let $\pi:\mathbb{P}^{2n}_{(u_0,\dots,u_{2n})}\dasharrow\mathbb{P}^{2n-1}_{(v_0,\dots,v_{2n-1})}$ be the projection from $[0:\dots:0:1]$. Consider the polynomials
$$
\begin{array}{lll}
\widetilde{A} & = \frac{\overline{A}+3\overline{B}}{4} & = u_0^3 + \sum_{i=0}^{n-1}u_{2i+1}^3+3u_{2i+1}u_{2i+2}^2; \\ 
\widetilde{B} & = \frac{\overline{A}-\overline{B}}{4}  & = \sum_{i=0}^{n-1}u_{2i+1}^2u_{2i+2}+3u_{2i+2}^3.
\end{array} 
$$
Then $\overline{A} = 0$ yields $u_0^3 + \sum_{i=0}^{n-2}(u_{2i+1}^3+3u_{2i+1}u_{2i+2}^2) + u_{2n-1}^3 = -3u_{2n-1}u_{2n}^2$ and hence
\stepcounter{thm}
\begin{equation}\label{sq}
u_{2n}^2 = - \frac{u_0^3 + \sum_{i=0}^{n-2}(u_{2i+1}^3+3u_{2i+1}u_{2i+2}^2) + u_{2n-1}^3}{3u_{2n-1}}.
\end{equation}
From $\overline{B} = 0$ we get
$$
\sum_{i=0}^{n-2}u_{2i+1}^2u_{2i+2}+3u_{2i+2}^3 = u_{2n-1}^2u_{2n}+3u_{2n}^3
$$
and hence
$$
\left(\sum_{i=0}^{n-2}u_{2i+1}^2u_{2i+2}+3u_{2i+2}^3\right)^2 = u_{2n-1}^4u_{2n}^2+6u_{2n-1}^2u_{2n}^4+9u_{2n}^6.
$$
Substituting the expression for $u_{2n}^2$ in (\ref{sq}) in this last equation we get
$$
3u_{2n-1}^3\left(\sum_{i=0}^{n-2}u_{2i+1}^2u_{2i+2}+3u_{2i+2}^3\right)^2 = -u_{2n-1}^6P+2u_{2n-1}^3P^2-P^3
$$
where $P = u_0^3+u_{2n-1}^3+\sum_{i=0}^{n-2}3u_{2i+1}u_{2i+2}^2+u_{2i+2}^3$. Therefore the projection $Z^{2n-2} = \pi(Y^{2n-2})$ is the hypersurface 
$$
Z^{2n-2} = \left\lbrace \overline{P}^3 -2v_{2n-1}^3\overline{P}^2+v_{2n-1}^6\overline{P} + 3v_{2n-1}^3\left(\sum_{i=0}^{n-2}v_{2i+1}^2v_{2i+2}+3v_{2i+2}^3\right)^2=0\right\rbrace\subset \mathbb{P}^{2n-1}_{(v_0,\dots,v_{2n-1})}
$$
where $\overline{P} = P(v_0,\dots,v_{2n-1})$. Since $Z^{2n-2}$ is irreducible to conclude it is enough to note that $\pi_{|Y^{2n-2}}:Y^{2n-2}\rightarrow Z^{2n-2}$ is birational.  
\end{proof}

\begin{Remark}\label{noIrrn1}
When $n = 1$ then $Y^{0}$ has the following three irreducible components:
$$
Y^0_1 = \{u_0 = u_1^2+3u_2^2 = 0\},\: Y^0_2 = \{u_2 = u_0^2-u_0u_1+u_1^2 = 0\},\: Y^0_3 = \{u_2 = u_0+u_1 = 0\}.
$$
Note that $Y^0_1$ and $Y^0_2$ consists of two pairs of points defined over a quadratic extension of the base field, while $Y^0_3$ is a point defined over the base field. 
\end{Remark}

\begin{Lemma}\label{lsing}
Consider the scheme $Z^{n-1}\subset\mathbb{P}^{2n}$ defined by the following equations:
$$
Z^{n-1} = \left\lbrace\begin{array}{ll}
u_0 & = 0;\\
u_{2i+1}^2+3u_{2i+2}^2 & = 0;\\
u_{2s+1}u_{2t+3}+3u_{2s+2}u_{2t+4} & = 0;\\
u_{2s+2}u_{2t+3}-u_{2s+1}u_{2t+4} & = 0;
\end{array}\right.
$$
for $i = 0,\dots,n-1$; $s,t = 0,\dots,n-2$ with $t\geq s$. Then $Z^{n-1}$ is smooth, 
$$
\dim(Z^{n-1}) = n-1,\quad \deg(Z^{n-1}) = 2
$$
and $Z^{n-1}$ splits as the disjoint union of two conjugate $(n-1)$-planes if and only if $-3$ is a square in $k$. Furthermore,
$$
\Sing(Y^{2n-2}) = Z^{n-1}
$$
and if $\ch(k) = 0$ then $Y^{2n-2}$ has multiplicity two in a general point of any of the geometric components of $Z^{n-1}$. 
\end{Lemma}
\begin{proof}
First, note that over the quadratic extension $k(\xi)$ with $\xi^2 = -3$ of the base field the scheme $Z^{n-1}$ splits as the disjoint union of the following linear spaces:
$$
Z_{+}^{n-1} = \{u_0 = u_{2i+1}+\xi u_{2i+2} = 0,\: i = 0,\dots n-1\}, \quad Z_{-}^{n-1} = \{u_0 = u_{2i+1}-\xi u_{2i+2} = 0,\: i = 0,\dots n-1\}.
$$
Hence, $\dim(Z^{n-1}) = n-1$, $\deg(Z^{n-1}) = 2$, and $Z^{n-1}$ is smooth. The partial derivatives of $\overline{A}$ are given by $\frac{\partial\overline{A}}{\partial u_0} = 3u_0^2$,
$$
\begin{array}{ll}
\frac{\partial\overline{A}}{\partial u_{2i+1}} & = 3(u_{2i+1}+u_{2i+2})^2;\\
\frac{\partial\overline{A}}{\partial u_{2i+2}} & = 3(u_{2i+1}^2+2u_{2i+1}u_{2i+2}+9u_{2i+2}^2);
\end{array} 
$$
for $i = 0,\dots,n-1$, and the partial derivatives of $\overline{B}$ are given by $\frac{\partial\overline{B}}{\partial u_0} = 3u_0^2$,
$$
\begin{array}{ll}
\frac{\partial\overline{B}}{\partial u_{2i+1}} & = 3u_{2i+1}^2-2u_{2i+1}u_{2i+2}+3u_{2i+2}^2;\\
\frac{\partial\overline{B}}{\partial u_{2i+2}} & = -(u_{2i+1}-3u_{2i+2})^2;
\end{array} 
$$
for $i = 0,\dots,n-1$. Therefore, the scheme $JY^{n-2}$, cut out by the $2\times 2$ minors of the corresponding Jacobian matrix, is defined by the following equations:
$$
\left\lbrace\begin{array}{ll}
u_0^2u_{2i+1}u_{2i+2} & = 0;\\ 
u_0^2(u_{2i+1}^2+9u_{2i+2}^2) & = 0;\\
(u_{2i+1}^2+3u_{2i+2}^2)^2 & = 0;\\
(u_{2s+1}u_{2t+4}-u_{2s+2}u_{2t+3})(u_{2s+2}u_{2t+3}-u_{2s+2}u_{2t+4}) & = 0;\\
(u_{2s+1}u_{2t+4}-u_{2s+2}u_{2t+3})(u_{2s+2}u_{2t+3}-9u_{2s+2}u_{2t+4}) & = 0;\\
u_{2s+1}^2u_{2t+3}^2+9u_{2s+1}^2u_{2t+4}^2-4u_{2s+1}u_{2s+2}u_{2t+3}u_{2t+4} + u_{2s+2}^2u_{2t+3}^2+9u_{2s+2}^2u_{2t+4}^2 & = 0;\\
u_{2s+1}^2u_{2t+3}^2+u_{2s+1}^2u_{2t+4}^2-4u_{2s+1}u_{2s+2}u_{2t+3}u_{2t+4} + 9u_{2s+2}^2u_{2t+3}^2+9u_{2s+2}^2u_{2t+4}^2 & = 0;
\end{array}\right.
$$
for $i = 0,\dots,n-1$; $s,t = 0,\dots,n-2$ with $t\geq s$. To get the equality $\Sing(Y^{2n-2}) = Z^{n-1}$ is it enough to note that the reduced subscheme of $JY^{n-2}$ is exactly $Z^{n-1}$.

Now, consider the point $p_{-} = [0:\xi:1:\xi:1:0:\dots :0]\in Z_{-}^{n-1}$, and let $TC_{p_{-}}Y^{2n-2}$ be the tangent cone of $Y^{2n-2}$ in $p_{-}$. Note that 
$$
TC_{p_{-}}Y^{2n-2}\cap \{u_5 = \dots = u_{2n} = 0\} = \{u_1-\xi u_2 +u_3-\xi u_4 = u_3^2-2\xi u_3u_4 -3u_4^2 = 0\}.
$$
Hence $\mult_{p_{-}}Y^{2n-2} = 2$, and the same holds for the conjugate point $p_{+} \in Z_{+}^{n-1}$.
\end{proof}

\begin{Remark}
The complete intersection $Y^{2n-2}$ has points of multiplicity bigger than two. For instance, when $n = 2$ the tangent cone of $Y^{2}$ at $q_{-} = [0:0:0:\xi:1]$ is
$$
TC_{q_{-}}Y^{2} = \{u_3-\xi u_4 = u_0^3+u_1^3-\xi u_1^2u_2+3u_1u_2^2-3\xi u_2^3 = 0\}
$$
and hence $\mult_{q_{-}}Y^{2} = 3$.
\end{Remark}

\begin{Proposition}\label{linsys4}
Let $\mathcal{F}_{2n}$ be the linear subsystem of $|\mathcal{O}_{\mathbb{P}^{2n}}(4)|$ spanned by the quartic polynomials $\overline{\varphi}_i$ in (\ref{comp_m}):
$$
\mathcal{F}_{2n} = \left\langle \overline{\varphi}_0,\dots,\overline{\varphi}_{2n+1}\right\rangle\subset |\mathcal{O}_{\mathbb{P}^{2n}}(4)|,
$$
and $|\mathcal{I}_{Y^{2n-2},2Z^{n-1}}(4)|\subset |\mathcal{O}_{\mathbb{P}^{2n}}(4)|$ the linear subsystem of $|\mathcal{O}_{\mathbb{P}^{2n}}(4)|$ of quartics containing $Y^{2n-2}$ and vanishing with multiplicity two on $Z^{n-1}$. Then
$$
\mathcal{F}_{2n} = |\mathcal{I}_{Y^{2n-2},2Z^{n-1}}(4)|
$$
and $h^0(\mathbb{P}^{2n},\mathcal{I}_{Y^{2n-2},2Z^{n-1}}(4)) = 2n+2$.
\end{Proposition}
\begin{proof}
Clearly the polynomials $\overline{\varphi}_i$ in (\ref{comp_m}) vanish on $Y^{2n-2} = \{\overline{A} = \overline{B} =0\}\subset\mathbb{P}^{2n}$. We will prove that the $\overline{\varphi}_i$ vanish with multiplicity at least two on $Z^{n-1}$. We have
$$
\begin{array}{ll}
\frac{\partial\overline{\varphi}_{2i}}{\partial u_0} = -3u_0^2(u_{2i+1}+3u_{2i+2}), & \frac{\partial\overline{\varphi}_{2i+1}}{\partial u_0} = 3u_0^2(u_{2i+1}-3u_{2i+2});
\end{array}
$$
for $i = 0,\dots,n-1$, and 
$$
\begin{array}{lll}
\frac{\partial\overline{\varphi}_{2n}}{\partial u_0} & = &  -4u_0^3 + \sum_{i=1}^{n-1}(-u_{2i+1}+3u_{2i+2})(u_{2i+1}^2+3u_{2i+2}^2);\\ 
\frac{\partial\overline{\varphi}_{2n+1}}{\partial u_0} & = & 4u_0^3 + \sum_{i=1}^{n-1}(u_{2i+1}+3u_{2i+2})(u_{2i+1}^2+3u_{2i+2}^2).
\end{array}
$$
Set 
$$
\alpha_{i,j} = u_{2j+1}^2+3u_{2j+2}^2,\: \beta_{i,j} = u_{2j+1}u_{2i+1}+3u_{2j+2}u_{2i+2},\: \gamma_{i,j} = u_{2j+2}u_{2i+1}-u_{2j+1}u_{2i+2} 
$$
and
$$
\begin{array}{ll}
\eta_j = u_{2j+1}^2+3u_{2j+2}^2, & \delta_i = u_{2i+1}^2+3u_{2i+2}^2,\\ 
\epsilon_{i,j} = u_{2i+1}u_{2j+1}+3u_{2i+2}u_{2j+2}, & \zeta_{i,j} = u_{2i+2}u_{2j+1}-u_{2i+1}u_{2j+2}. 
\end{array} 
$$
First, consider the polynomials $\overline{\varphi}_{2i}$ with $i = 1,\dots,n-1$. We have
$$
\begin{array}{lll}
\frac{\partial\overline{\varphi}_{2i}}{\partial u_{2j+1}} & = & -3(u_{2i+1}+3u_{2i+2})\alpha_{i,j}+6u_{2j+2}\beta_{i,j}+6u_{2j+2}\gamma_{i,j};\\ 
\frac{\partial\overline{\varphi}_{2i}}{\partial u_{2j+2}} & = & 3(u_{2i+1}-u_{2i+2})\alpha_{i,j}-6u_{2j+2}\beta_{i,j}+18u_{2j+2}\gamma_{i,j};
\end{array} 
$$
for $0\leq j < i$,
$$
\begin{array}{lll}
\frac{\partial\overline{\varphi}_{2i}}{\partial u_{2i+1}} & = & \sum_{j=0}^{i-1}(-u_{2j+1}+3u_{2j+2})\eta_j-4u_{2i+1}\delta_i-\sum_{j=i+1}^{n-1}(u_{2j+1}-3u_{2j+2})\eta_j - u_0^3;\\ 
\frac{\partial\overline{\varphi}_{2i}}{\partial u_{2i+2}} & = & -3\sum_{j=0}^{i-1}(u_{2j+1}+u_{2j+2})\eta_j -12u_{2i+2}\delta_i-3\sum_{j=i+1}^{n-1}(u_{2j+1}+u_{2j+2})\eta_j - 3u_0^3;
\end{array} 
$$
and 
$$
\begin{array}{lll}
\frac{\partial\overline{\varphi}_{2i}}{\partial u_{2j+1}} & = & -3(u_{2j+1}+u_{2j+2})\epsilon_{i,j}-3(3u_{2j+1}+u_{2j+2})\zeta_{i,j};\\ 
\frac{\partial\overline{\varphi}_{2i}}{\partial u_{2j+2}} & = & 3(u_{2j+1}-3u_{2j+2})\epsilon_{i,j}-3(u_{2j+1}+9u_{2j+2})\zeta_{i,j};
\end{array} 
$$
for $i < j \leq n-1$. For $i = n$ we get
$$
\begin{array}{lll}
\frac{\partial\overline{\varphi}_{2n}}{\partial u_{2j+1}} & = &  -3u_0(u_{2j+1}-u_{2j+2})^2;\\ 
\frac{\partial\overline{\varphi}_{2n}}{\partial u_{2j+2}} & = & 3u_0(u_{2j+1}^2-2u_{2j+1}u_{2j+2}+9u_{2j+2}^2);
\end{array}
$$
for $j = 0,\dots,n-1$.

Now, consider the polynomials $\overline{\varphi}_{2i+1}$ with $i = 1,\dots,n-1$. We have
$$
\begin{array}{lll}
\frac{\partial\overline{\varphi}_{2i+1}}{\partial u_{2j+1}} & = & 3(u_{2i+1}-3u_{2i+2})\alpha_{i,j}+6u_{2j+2}\beta_{i,j}-6u_{2j+2}\gamma_{i,j};\\ 
\frac{\partial\overline{\varphi}_{2i+1}}{\partial u_{2j+2}} & = & 3(u_{2i+1}+u_{2i+2})\alpha_{i,j}+6u_{2j+2}\beta_{i,j}+18u_{2j+2}\gamma_{i,j};
\end{array} 
$$
for $0\leq j < i$, 
$$
\begin{array}{lll}
\frac{\partial\overline{\varphi}_{2i+1}}{\partial u_{2i+1}} & = & \sum_{j=0}^{i-1}(u_{2j+1}+3u_{2j+2})\eta_j+4u_{2i+1}\delta_i+\sum_{j=i+1}^{n-1}(u_{2j+1}+3u_{2j+2})\eta_j + u_0^3;\\ 
\frac{\partial\overline{\varphi}_{2i+1}}{\partial u_{2i+2}} & = & -3\sum_{j=0}^{i-1}(u_{2j+1}-u_{2j+2})\eta_j +12u_{2i+2}\delta_i-3\sum_{j=i+1}^{n-1}(u_{2j+1}-u_{2j+2})\eta_j - 3u_0^3;
\end{array} 
$$
and 
$$
\begin{array}{lll}
\frac{\partial\overline{\varphi}_{2i+1}}{\partial u_{2j+1}} & = & 3(u_{2j+1}-u_{2j+2})\epsilon_{i,j}-3(3u_{2j+1}+u_{2j+2})\zeta_{i,j};\\ 
\frac{\partial\overline{\varphi}_{2i+1}}{\partial u_{2j+2}} & = & 3(u_{2j+1}+3u_{2j+2})\epsilon_{i,j}+3(u_{2j+1}-9u_{2j+2})\zeta_{i,j};
\end{array} 
$$
for $i < j \leq n-1$. For $i = n$ we get
$$
\begin{array}{lll}
\frac{\partial\overline{\varphi}_{2n+1}}{\partial u_{2j+1}} & = &  3u_0(u_{2j+1}+u_{2j+2})^2;\\ 
\frac{\partial\overline{\varphi}_{2n+1}}{\partial u_{2j+2}} & = & 3u_0(u_{2j+1}^2+2u_{2j+1}u_{2j+2}+9u_{2j+2}^2);
\end{array}
$$
for $j = 0,\dots,n-1$. 

Therefore, by the definition of $Z^{n-1}$ in Lemma \ref{lsing} we get that the $\overline{\varphi}_i$ vanish with multiplicity at least two on $Z^{n-1}$. Note that $\frac{\partial^2\overline{\varphi}_{2n+1}}{\partial u_{0}\partial u_{2j+1}}$ does not vanish along $Z^{n-1}$, and hence a general element of $\mathcal{F}_{2n}$ has multiplicity two along $Z^{n-1}$. 

So far we proved that $\mathcal{F}_{2n} \subset |\mathcal{I}_{Y^{2n-2},2Z^{n-1}}(4)|$. It is clear from their expression in (\ref{ABu0}) that the $\overline{\varphi}_i$ for $i = 0,\dots, 2n+1$ form a basis of $\mathcal{F}_{2n}$. Therefore, to conclude it is enough to show that $h^0(\mathbb{P}^{2n},\mathcal{I}_{Y^{2n-2},2Z^{n-1}}(4)) = 2n+2$.

Let $F\in k[u_0,\dots,u_{2n}]_4$ be a homogeneous polynomial of degree four vanishing on $Y^{2n-2}$. Then 
$$
F = L_{\overline{A}}\overline{A} + L_{\overline{B}}\overline{B}
$$
with 
$$
\begin{array}{ll}
L_{\overline{A}} = \sum_{i=0}^{2n}a_iu_i, & L_{\overline{B}} = \sum_{i=0}^{2n}b_iu_i;
\end{array}
$$
linear forms. In particular, the space $H^0(\mathbb{P}^{2n},\mathcal{I}_{Y^{2n-2}}(4))$ of quartics containing $Y^{2n-2}$ has dimension $4n+2$.

Consider the geometric components $Z^{n-1}_{+}$ and $Z^{n-1}_{-}$ of $Z^{n-1}$ in the proof of Lemma \ref{lsing}. On $Z^{n-1}_{+}$ we have $u_0=0$ and $u_{2i+1} = -\xi u_{2i+2}$ for $i = 0,\dots,n-1$, while on $Z^{n-1}_{-}$ we have $u_0=0$ and $u_{2i +1} = \xi u_{2i+2}$ for $i = 0,\dots,n-1$. Set
$$
\begin{array}{lll}
L_{+} & = & \sum_{i=0}^{n-1}(\xi a_{2i+2}-\xi b_{2i+1}+3a_{2i+1}+b_{2i+2})u_{2i+2};\\
L_{-} & = & \sum_{i=0}^{n-1}(\xi a_{2i+2}-\xi b_{2i+1}-3a_{2i+1}-b_{2i+2})u_{2i+2}.
\end{array}
$$
Taking the partial derivatives of $F$ and substituting first $u_0=0$, $u_{2i+1} = -\xi u_{2i+2}$ and then $u_0=0$, $u_{2i+1} = \xi u_{2i+2}$ we get that $\frac{\partial F}{\partial u_{0}}_{|Z^{n-1}} = 0$, and
$$
\begin{array}{ll}
\frac{\partial F}{\partial u_{2j+1}}_{|Z^{n-1}_{+}} =  2(\xi - 3)u_{2j+2}^2 L_{+},  & \frac{\partial F}{\partial u_{2j+2}}_{|Z^{n-1}_{+}} = -6(\xi + 1)u_{2j+2}^2 L_{+};\\
\frac{\partial F}{\partial u_{2j+1}}_{|Z^{n-1}_{-}} = 2(\xi + 3)u_{2j+2}^2 L_{-}, & \frac{\partial F}{\partial u_{2j+2}}_{|Z^{n-1}_{-}} = -6(\xi - 1)u_{2j+2}^2 L_{-};
\end{array} 
$$
for $j \in 0,\dots,n-1$.

Therefore, the subspace of $H^0(\mathbb{P}^{2n},\mathcal{I}_{Y^{2n-2}}(4))$ consisting of quartics which are singular along $Z^{n-1}$ is defined by
$$
\left\lbrace\begin{array}{lll}
\xi (a_{2i+2}-b_{2i+1})+3a_{2i+1}+b_{2i+2} & = & 0;\\
\xi (a_{2i+2}- b_{2i+1})-3a_{2i+1}-b_{2i+2} & = & 0;
\end{array}\right.
$$ 
for $i = 0,\dots,n-1$, that is $b_{2i+1} = a_{2i+2}$ and $b_{2i+2} = -3a_{2i+1}$ for $i = 0,\dots,n-1$. Finally, 
$$H^0(\mathbb{P}^{2n},\mathcal{I}_{Y^{2n-2},2Z^{n-1}}(4))\subset H^0(\mathbb{P}^{2n},\mathcal{I}_{Y^{2n-2}}(4))$$ 
has dimension $4n+2 - 2n = 2n+2$.
\end{proof}

\section{Rationality in characteristic two}\label{ch2}
So far we proved that the linear system $|\mathcal{I}_{Y^{2n-2},2Z^{n-1}}(4)|\subset |\mathcal{O}_{\mathbb{P}^{2n}}(4)|$ of quartics which contains $Y^{2n-2}$ and are singular along $Z^{n-1}$ yields a rational parametrization 
\stepcounter{thm}
\begin{equation}\label{par_rs}
\overline{\varphi}_{2n}:\mathbb{P}^{2n}\dasharrow X^{2n}\subset\mathbb{P}^{2n+1}
\end{equation}
of the Fermat cubic hypersurface $X^{2n}$. From the computations in Section \ref{res_scal} it is clear that $\overline{\varphi}_{2n}$ can not yield a birational parametrization in characteristic two. One could try to take the reduction modulo two of the polynomials defining $\overline{\varphi}_{2n}$ but this would lead to a map contracting $\mathbb{P}^{2n}$ onto the linear subspace 
$$\{x_{2i}+x_{2i+1} = 0,\text{ for } i = 0,\dots,n\}\subset X^{2n}.$$ 

Hence, a different idea is needed. An explicit description of the inverse of $\overline{\varphi}_{2n}$ will lead us to a birational parametrization of $X^{2n}$ in characteristic two. We will not develop in full detail the computations that are similar to those in Section \ref{res_scal}. Set 
$$
P = \sum_{i=0}^{2n}u_i^3,\: Q = \sum_{i=0}^{n-1}u_{2i}^2u_{2i+1}+u_{2i}u_{2i+1}^2+u_{2i+1}^3
$$
and\\
\stepcounter{thm}
\begin{equation}\label{eqg}
\begin{array}{lll}
g_{2i} & = & (u_{2i}+u_{2i+1})P + u_{2i}Q; \\ 
g_{2i+1} & = & u_{2i}P + u_{2i+1}Q; \\ 
g_{2n} & = & u_{2n}(P + Q); \\ 
g_{2n+1} & = & u_{2n}P; 
\end{array}
\end{equation}
for $i = 0,\dots,n-1$.

\begin{Lemma}\label{leminv0}
Consider the subschemes
$$
Y^{2n-2} = \{P = Q = 0\}
$$
and
$$
Z^{n-1} = 
\left\lbrace
\begin{array}{llll}
u_{2i}^2+u_{2i}u_{2i+1}+u_{2i+1}^2 & = & 0 & \text{for } i = 0,\dots,n-1; \\ 
u_{2i}u_{2j}+u_{2i}u_{2j+1}+u_{2i+1}u_{2j+1} & = & 0 & \text{for } i = 0,\dots,n-1,\: j = i+1\dots,n-1; \\ 
u_{2i+1}u_{2j+2}+u_{2i}u_{2j+3} & = & 0 & \text{for } i = 0,\dots,n-1,\: j = i,\dots,n-2;\\ 
u_{2n} & = & 0; & 
\end{array}\right.
$$
in $\mathbb{P}^{2n}$. Then $Y^{2n-2}$ is irreducible of dimension $2n-2$ and degree nine, $Z^{n-1}$ has degree two and geometrically it is the union of two conjugate $(n-1)$-planes, and $\Sing(Y^{2n-2}) = Z^{n-1}$.
\end{Lemma}
\begin{proof}
The statement can be proved by computations analogous to those in the proofs of Lemma \ref{comp_m} and of Lemma \ref{lsing}.
\end{proof}

\begin{Proposition}\label{pinv1}
Consider the linear subsystems
$$
\mathcal{C}_{2n} = \left\langle g_0,\dots,g_{2n+1}\right\rangle \subset |\mathcal{O}_{\mathbb{P}^{2n}}(4)|
$$
and $|\mathcal{I}_{Y^{2n-2},2Z^{n-1}}(4)|\subset |\mathcal{O}_{\mathbb{P}^{2n}}(4)|$ of quartics containing $Y^{2n-2}$ and which are singular along $Z^{n-1}$. Then $\mathcal{C}_{2n} = |\mathcal{I}_{Y^{2n-2},2Z^{n-1}}(4)|\subset |\mathcal{O}_{\mathbb{P}^{2n}}(4)|$.
\end{Proposition}
\begin{proof}
The statement follows from computations similar to those in the proof of Proposition \ref{linsys4}. 
\end{proof}

\begin{Proposition}\label{inv1}
Consider the following subschemes 
$$
H^{n}_{\pm} = 
\left\lbrace\begin{array}{llll}
x_{2i}^2-x_{2i}x_{2i+1}+x_{2i+1}^2 & = & 0 & \text{for } i = 0,\dots,n; \\ 
x_{2i}x_{2j+2}-x_{2i}x_{2j+3}+x_{2i+1}x_{2j+3} & = & 0 & \text{for } i = 0,\dots,n-1,\: j = i,\dots,n-1; \\ 
x_{2i-1}x_{2j}-x_{2i-2}x_{2j+1} & = & 0 & \text{for } i = 1,\dots,n,\: j = i,\dots,n;
\end{array}\right.
$$
$$
X^{2n-2} = X^{2n} \cap \{x_{2n} = x_{2n+1} = 0\}
$$
in $\mathbb{P}^{2n+1}$. The linear system $|\mathcal{I}_{H^{n}_{\pm},X^{2n-2}}|\subset |\mathcal{O}_{\mathbb{P}^{2n+1}}(2)|$ induces a rational map $\mathbb{P}^{2n+1}\dasharrow \mathbb{P}^{2n}$ whose restriction $\qoppa_{2n}:X^{2n}\dasharrow \mathbb{P}^{2n}$ to $X^{2n}$ is birational. 
\end{Proposition}
\begin{proof}
Note that geometrically $H^{n}_{\pm}$ is the union of two skew conjugate $n$-planes contained in $X^{2n}$. Hence, the linear system of quadric hypersurfaces of $\mathbb{P}^{2n+1}$ containing $H^{n}_{\pm}$ has $(n+1)^2$ sections. The containment of $X^{2n-2}$ imposes $n^2-1$ further independent condition on these quadrics and hence $|\mathcal{I}_{H^{n}_{\pm},X^{2n-2}}|$ has $2n+1$ sections that we will now write down explicitly. The polynomials
\begin{small}
$$
\begin{array}{lll}
\Qoppa_{2i} & = & x_{2i}x_{2n}-x_{2i}x_{2n+1} + x_{2i+1}x_{2n+1}; \\ 
\Qoppa_{2i+1} & = & x_{2i+1}x_{2n} - x_{2i}x_{2n+1}; \\ 
\Qoppa_{2n} & = & x_{2n}^2 - x_{2n}x_{2n+1} + x_{2n+1}^2; 
\end{array}
$$
\end{small}
for $i = 0,\dots,n-1$ form a basis for the space of sections of $|\mathcal{I}_{H^{n}_{\pm},X^{2n-2}}|$. Consider the map  
$$
\begin{array}{cccc}
 \qoppa_{2n}: & X^{2n}\subset\mathbb{P}^{2n+1} & \dasharrow & \mathbb{P}^{2n}_{(u_0,\dots,u_{2n}}\\
  & (x_{0},\dots,x_{2n+1}) & \mapsto & [\Qoppa_0:\dots:\Qoppa_{2n}].
\end{array}
$$
First, assume that $\ch(k)\neq 2$ and let $h:\mathbb{P}^{2n}_{(u_0,\dots,u_{2n})}\rightarrow \mathbb{P}^{2n}_{(u_0,\dots,u_{2n})}$ be the automorphism of $\mathbb{P}^{2n}_{(u_0,\dots,u_{2n})}$ given by 
$$
u_{2i+1}\mapsto 2u_{2i}-u_{2i+1},\: u_{2i+2}\mapsto u_{2i+1} 
$$
for $i = 0,\dots,n-1$, and $u_{0}\mapsto 2u_{2n}$. Then by (\ref{comp_m}) we get that $h\circ \qoppa_{2n}: X^{2n}\dasharrow \mathbb{P}^{2n}_{(u_0,\dots,u_{2n})}$ is the birational inverse of $\overline{\varphi}_{2n}$. 

Now, let $\ch(k) = 2$. A standard computation shows that the map
$$
\begin{array}{cccc}
 g_{2n}: & \mathbb{P}^{2n}_{(u_0,\dots,u_{2n})} & \dasharrow & X^{2n} \subset\mathbb{P}^{2n+1}\\
  & (u_{0},\dots,u_{2n}) & \mapsto & [g_0:\dots:g_{2n+1}]
\end{array}
$$
given by the quartics in (\ref{eqg}) is the birational inverse of $\qoppa_{2n}$.
\end{proof}

\begin{Remark}\label{sumup}
Summing-up we proved that the rational map $\qoppa_{2n}:X^{2n}\dasharrow \mathbb{P}^{2n}$ is birational over any field $k$ with $\ch(k)\neq 3$. When $\ch(k)\neq 2$ then the rational map $\overline{\varphi}_{2n}: \mathbb{P}^{2n}\dasharrow X^{2n}$ given by the restriction of scalars is, up to the automorphism $h:\mathbb{P}^{2n}\rightarrow\mathbb{P}^{2n}$, the birational inverse of $\qoppa_{2n}$. However, this is not the case in characteristic two. Indeed, when $\ch(k) = 2$ the map $\overline{\varphi}_{2n}$ contracts $\mathbb{P}^{2n}$ onto a proper subvariety of $X^{2n}$ and we must come up with an alternative rational map. This is the map $g_{2n}: \mathbb{P}^{2n}\dasharrow X^{2n}$ which by Proposition \ref{inv1} is the birational inverse of $\qoppa_{2n}$ in characteristic two. 
\end{Remark}

\section{Quadro-Cubic Cremona transformations}\label{bri}

We will now work out explicitly the birational parametrization $\phi:\mathbb{G}(1,n+1)\dasharrow X^{2n}$ in the proof of Proposition \ref{p1}. As in Section \ref{ch2} for the computations that are analogous to those in Section \ref{res_scal} we will be quick. In the notation of the proof of Proposition \ref{p1} set
$$
\Lambda = \{x_1 = x_{2i} = 0, \text{ for } i = 0,\dots,n\}\subset\mathbb{P}^{2n+1}
$$
and write a general $(n+1)$-plane $H\subset\mathbb{P}^{2n+1}$ containing $\Lambda$ as 
$$
H =  \{x_{2i} + t_{2i+1}x_1 + t_{2i+2}x_{2n} = 0, \text{ for } i\in [0,n-1]\}\subset\mathbb{P}^{2n+1}
$$
with $t_{2i+1},t_{2i+2}\in k$. Note that $\Lambda\cap (H_{+}\cup H_{-}) = \emptyset$. Let $p^{a_{+}}$ be the intersection point of $H$ and $H_{+}$, and $p^{a_{-}}$ the intersection point of $H$ and $H_{-}$. We have
\stepcounter{thm}
\begin{equation}\label{eq+}
\begin{array}{lll}
p^{a_{+}}_0 & = & -\frac{(\xi+1)^2t_{2}}{2(\xi+2t_{1}+1)}; \\ 
p^{a_{+}}_1 & = & -\frac{(\xi+1)t_{2}}{\xi+2t_{1}+1}; \\ 
p^{a_{+}}_{2i+2} & = & -\frac{(\xi+1)(\xi t_{2i+4}+2t_{1}t_{2i+4}-2t_{2i+3}t_{2}+t_{2i+4})}{2(\xi+2t_{1}+1)}; \\ 
p^{a_{+}}_{2i+3} & = & -\frac{\xi(t_{1}t_{2i+4}-t_{2i+3}t_{2}+t_{2i+4})-t_{1}t_{2i+4}+t_{2i+3}t_{2}+t_{2i+4}}{\xi (t_{1}+1)+t_{1}-1}; 
\end{array} 
\end{equation}
for $i = 0,\dots,n-2$, $p^{a_{+}}_{2n} = \frac{1+\xi}{2}$, and 
\stepcounter{thm}
\begin{equation}\label{eq-}
\begin{array}{lll}
p^{a_{-}}_0 & = & \frac{(\xi-1)^2t_{2}}{2(\xi-2t_{1}-1)}; \\ 
p^{a_{-}}_1 & = & -\frac{(\xi-1)t_{2}}{\xi-2t_{1}-1}; \\ 
p^{a_{-}}_{2i+2} & = & \frac{(\xi-1)(\xi t_{2i+4}-2t_{1}t_{2i+4}+2t_{2i+3}t_{2}-t_{2i+4})}{2(\xi-2t_{1}-1)}; \\ 
p^{a_{-}}_{2i+3} & = & -\frac{\xi(t_{1}t_{2i+4}-t_{2i+3}t_{2}+t_{2i+4})+t_{1}t_{2i+4}-t_{2i+3}t_{2}-t_{2i+4}}{\xi (t_{1}+1)-t_{1}+1}; 
\end{array} 
\end{equation}
for $i = 0,\dots,n-2$, $p^{a_{-}}_{2n} = \frac{1-\xi}{2}$. Now, consider the line
$$
L = \left\langle p^{a_{+}},p^{a_{-}}\right\rangle
$$
and write parametrically
$$
L_i = p^{a_{+}}_i+\lambda(p^{a_{-}}_i-p^{a_{+}}_i)
$$
with $\lambda\in k$, $i = 0,\dots,2n$. Substituting $x_i = L_i$ in $F = x_0^3+\dots+x_{2n}^3+1$ we get a polynomial $F(\lambda)$ having three roots:
$$
\lambda_1 = 0,\: \lambda_2 = 1,\: \lambda_3 = \frac{N_{2n}\xi + M_{2n}}{D_{2n}}.
$$
We will now describe the polynomials $N_{2n},M_{2n},D_{2n}$. We have
$$
\begin{footnotesize}
\begin{array}{lll}
D_2 & = & 2(t_{1}^4 - 2t_{1}t_{2}^3 + 2t_{1}^3 - t_{2}^3 + 3t_{1}^2 + 2t_{1} + 1);\\
D_4 & = & 2(t_1^4t_4^3 - 3t_1^3t_2t_3t_4^2 + 3t_1^2t_2^2t_3^2t_4 - t_1t_2^3t_3^3 + 2t_1^3t_4^3 - 3t_1^2t_2t_3t_4^2 + t_2^3t_3^3 + 3t_1^2t_4^3 - 3t_1t_2t_3t_4^2 - t_1^4+ 2t_1t_2^3 + 2t_1t_4^3 - 2t_1^3 + t_2^3 \\ 
 &  &  + t_4^3- 3t_1^2 - 2t_1 - 1).
\end{array} 
\end{footnotesize}
$$
Set
$$
\begin{footnotesize}
\begin{array}{lll}
D_{2i}' & = & 2(t_{1}^4t_{2i}^3 - 3t_{1}^3t_{2i-1}t_{2}t_{2i}^2 + 3t_{1}^2t_{2i-1}^2t_{2}^2t_{2i} - t_{1}t_{2i-1}^3t_{2}^3 + 2t_{1}^3t_{2i}^3 - 3t_{1}^2t_{2i-1}t_{2}t_{2i}^2+ t_{2i-1}^3t_{2}^3 + 3t_{1}^2t_{2i}^3 - 3t_{1}t_{2i-1}t_{2}t_{2i}^2\\
 & & + 2t_{1}t_{2i}^3 + t_{2i}^3 ).
\end{array} 
\end{footnotesize}
$$
Then $D_{2n}$ is given recursively by the formula
\stepcounter{thm}
\begin{equation}\label{den}
D_{2i} = D_{2i}' + D_{2(i-1)}
\end{equation}
for $3\leq i\leq n$. Furthermore
$$
\begin{footnotesize}
\begin{array}{lll}
N_2 & = & -t_{1}^4 - 2t_{1}^3 + t_{2}^3 - 3t_{1}^2 - 2t_{1} - 1;\\ 
M_2 & = & t_{1}^4 - 2t_{1}t_{2}^3 + 2t_{1}^3 - t_{2}^3 + 3t_{1}^2 + 2t_{1} + 1;\\
N_4 & = & 3t_1^2 + t_1^4 + 2t_1 - t_2^3 + 2t_1^3 - t_4^3 - 2t_2^2t_3^2t_4 + 2t_2t_3t_4^2 + t_1t_2^3t_3^3 - 4t_1t_2^2t_3^2t_4 - 3t_1^2t_2^2t_3^2t_4+ 5t_1^2t_2t_3t_4^2 + 5t_1t_2t_3t_4^2 + 3t_1^3t_2t_3t_4^2  \\ 
 & & + t_2^3t_3^3 - t_1^4t_4^3- 2t_1^3t_4^3 - 3t_1^2t_4^3 - 2t_1t_4^3 + 1;\\
M_4 & = & -t_1t_2^3t_3^3 - t_1^4 + 2t_1t_2^3 - 2t_1^3 + t_2^3 - 3t_1^2 - 2t_1 + t_1^4t_4^3 + 2t_1^3t_4^3 + t_2^3t_3^3 + 3t_1^2t_4^3 + 2t_1t_4^3- 3t_1^3t_2t_3t_4^2 + 3t_1^2t_2^2t_3^2t_4 - 3t_1^2t_2t_3t_4^2\\ 
 & &  - 3t_1t_2t_3t_4^2 + t_4^3-1.
\end{array} 
\end{footnotesize}
$$
Set
$$
\begin{footnotesize}
\begin{array}{lll}
N_{2i}' & = & t_2^3t_{2i-1}^3-2t_1^3t_{2i}^3 - t_1^4t_{2i}^3 - 2t_1t_{2i}^3 - 3t_1^2t_{2i}^3 - t_{2i}^3 + t_1t_2^3t_{2i-1}^3 + 2t_2t_{2i-1}t_{2i}^2 - 2t_2^2t_{2i-1}^2t_{2i}+ 3t_1^3t_2t_{2i-1}t_{2i}^2 + 5t_1^2t_2t_{2i-1}t_{2i}^2  \\
 & &  + 5t_1t_2t_{2i-1}t_{2i}^2 - 3t_1^2t_2^2t_{2i-1}^2t_{2i} - 4t_1t_2^2t_{2i-1}^2t_{2i};\\ 
M_{2i}' & = & t_{2i}^3 - 3t_1^3t_2t_{2i-1}t_{2i}^2 + 3t_1^2t_2^2t_{2i-1}^2t_{2i} - t_1t_2^3t_{2i-1}^3 - 3t_1^2t_2t_{2i-1}t_{2i}^2 - 3t_1t_2t_{2i-1}t_{2i}^2 + t_1^4t_{2i}^3+ 2t_1^3t_{2i}^3 + t_2^3t_{2i-1}^3 \\
 & & + 3t_1^2t_{2i}^3 + 2t_1t_{2i}^3.
\end{array} 
\end{footnotesize}
$$
Then $N_{2n}$ and $M_{2n}$ are given recursively as follows
\stepcounter{thm}
\begin{equation}\label{num}
\begin{array}{l}
N_{2i} = N_{2i}' + N_{2(i-1)}; \\ 
M_{2i} = M_{2i}' + M_{2(i-1)};
\end{array} 
\end{equation}
for $3\leq i\leq n$.

\begin{Proposition}\label{deg8}
Let $\overline{\phi}_{2n}:\mathbb{P}^{2n}\dasharrow X^{2n}\subset\mathbb{P}^{2n+1}$ be the map induced by the parametrization
$$
\begin{array}{cccc}
 \widetilde{\phi}: & \mathbb{A}^{2n} & \dasharrow & X^{2n}\subset\mathbb{P}^{2n+1}\\
  & (t_{1},\dots,t_{2n}) & \mapsto & [\widetilde{\phi}_0:\dots:\widetilde{\phi}_{2n+1}]
\end{array}
$$
where 
$$
\widetilde{\phi}_i = p^{a_{+}}_i+\frac{N_{2n}\xi + M_{2n}}{D_{2n}}(p^{a_{-}}_i-p^{a_{+}}_i)
$$ 
for $i = 0,\dots,2n$, and $\widetilde{\phi}_{2n+1} = 1$. Denote by $\mathcal{L}_{\overline{\phi}_{2n}} = |\overline{\phi}_{2n}^{*}\mathcal{O}_{X^{2n}}(1)|\subset|\mathcal{O}_{\mathbb{P}^{2n}}(d)|$ the linear system associated to $\overline{\phi}_{2n}$. Then $\overline{\phi}_{2n}$ is birational, and $d = 4$ when $n = 1$ while for $n\geq 2$ we have $d = 8$.
\end{Proposition}
\begin{proof}
The affine space $\mathbb{A}^{2n}$ with coordinates $(t_{1},\dots,t_{2n})$ is nothing but an affine chart of the Grassmannian $\mathbb{G}(1,n+1)$, and also a chart of the projective space $\mathbb{P}^{2n}$ with homogeneous coordinates $[t_0:\dots :t_{2n}]$. By the construction in the first part of Section \ref{bri} we have the following commutative diagram
$$
\begin{tikzcd}
\mathbb{A}^{2n} \arrow[rr, hook] \arrow[d, hook] \arrow[rrd, "\widetilde{\phi}", dashed] &  & {\mathbb{G}(1,n+1)} \arrow[d, "\phi", dashed] \\
\mathbb{P}^{2n} \arrow[rr, "\overline{\phi}_{2n}", dashed]                                    &  & X^{2n}                                       
\end{tikzcd}
$$
and hence the birationality of $\overline{\phi}_{2n}$ follows from that of $\phi$ which in turn comes from Proposition \ref{p1}.

First consider the case $n\geq 2$. Note that by (\ref{eq+}), (\ref{eq-}), (\ref{den}), (\ref{num}) the polynomial $t_{1}^2+t_{1}+1$ is a factor of both the numerator and the denominator of $\widetilde{\phi}_i$ for $i = 0,\dots,2n-1$. Hence, after clearing this common factor we have $\widetilde{\phi}_i = \frac{\alpha_i}{\beta_i}$ where $\deg(\alpha_1) = \deg(\beta_1) = \deg(\alpha_1) = \deg(\beta_1) = 7$, $\deg(\alpha_j) = 8$ and $\deg(\beta_j) = 7$ for $j = 2,\dots,2n-1$. Furthermore, $\deg(\alpha_{2n}) = \deg(\beta_{2n}) = 7$ and $\beta_i = \frac{D_{2n}}{2}$ for $i = 0,\dots,2n$. To get the expression of $\overline{\phi}_{2n}$ we multiply the vector $(\widetilde{\phi}_0,\dots,\widetilde{\phi}_{2n},1)$ by $\frac{D_{2n}}{2}$ and homogeneize, using the new variable $t_0$, the vector so obtained in order to get a vector whose entries are homogeneous polynomials of degree eight. 

For $n = 1$ we have that $\deg(\alpha_i) = \deg(\beta_i) = 4$ for $i = 0,1,2$, and hence $\overline{\phi}_{2n}$ is induced by a linear system of quartics.
\end{proof}

The geometric construction in Section \ref{gA} and the algebraic construction in Section \ref{res_scal} are strongly related. After all both of them boil down to taking a line through a pair of points lying in disjoint conjugate linear subspaces. We will now make this statement precise by introducing special Cremona transformations of $\mathbb{P}^{2n}$. 

Consider the following subschemes
$$
\begin{array}{llll}
T_1 & = & \{t_0^2+t_0t_1+t_1^2 = t_2 = 0\}; & \\ 
T_2 & = & \{t_0 = t_1 = 0\}; & \\ 
T_3 & = & \{t_0 = t_{2i+1} = 0\} & \text{for } i = 0,\dots,n-1;\\ 
T_4 & = & \{t_0 = t_{2i}t_{2j+1}-t_{2i-1}t_{2j+2} = 0\} & \text{for } i = 1,\dots,n-1;\: i\leq j\leq n-1;
\end{array} 
$$
in $\mathbb{P}^{2n}_{(t_0,\dots,t_{2n})}$, and the following subschemes
$$
\begin{array}{llll}
U_1 & = & \{u_1 = u_2 = 0\};
\end{array} 
$$
and
$$
U_2 = \left\lbrace
\begin{array}{llll}
u_0 & = & 0; &  \\ 
u_{2i+1}^2 + 3u_{2i+2}^2 & = & 0 & \text{for } i = 0,\dots,n;\\ 
u_{2i+1}u_{2j+1} + 3u_{2i+2}u_{2j+2} & = & 0 & \text{for } i = 0,\dots,n-2;\: i < j \leq n-1;\\ 
u_{2i}u_{2j+1}-u_{2i-1}u_{2j+1} & = & 0 & \text{for } i = 1,\dots,n-1;\: i \leq j \leq n-1;
\end{array}\right.
$$
in $\mathbb{P}^{2n}_{(u_0,\dots,u_{2n})}$. Furthermore, set
\stepcounter{thm}
\begin{equation}\label{alp}
\begin{array}{lll}
\alpha_0 & = & -2t_0^3 - 2t_0^2t_1 - 2t_0t_1^2;\\
\alpha_1 & = &  2t_0^2t_2 + t_0t_1t_2;\\
\alpha_{2i} & = & t_0t_2t_{2i-1};\\       
\alpha_{2j+1} & = & -t_0t_2t_{2j+1}-2t_1t_2t_{2j+1}+2t_0^2t_{2j+2}+2t_1^2t_{2j+2}+2t_0t_1t_{2j+2}; 
\end{array}
\end{equation}
for $i = 1,\dots,n$, $j = 1,\dots,n-1$ and
\stepcounter{thm}
\begin{equation}\label{bet}
\begin{array}{lll}
\beta_0 & = & -u_0u_1 + u_0u_2;\\
\beta_{2i+1} & = & -2u_0u_{2i+2};\\
\beta_2 & = & u_1^2 + 3u_2^2;\\
\beta_{2j+4} & = & u_1u_{2j+3} - u_2u_{2j+3} + u_1u_{2j+4} + 3u_2u_{2j+4};
\end{array}
\end{equation}
for $i = 0,\dots,n-1$; $j = 0,\dots,n-2$.

\begin{Proposition}\label{crem}
Let $n\geq 2$ and set 
$$
\mathcal{T}_{2n} = \left\langle \alpha_0,\dots,\alpha_{2n}\right\rangle \subset |\mathcal{O}_{\mathbb{P}^{2n}_{(t_0,\dots,t_{2n})}}(3)| \text{ and } \mathcal{U}_{2n} = \left\langle \beta_0,\dots,\beta_{2n}\right\rangle \subset |\mathcal{O}_{\mathbb{P}^{2n}_{(u_0,\dots,u_{2n})}}(2)|.
$$ 
Then 
$$
\mathcal{T}_{2n} = |\mathcal{I}_{T_1,T_2,2T_3,T_4}(3)|\subset |\mathcal{O}_{\mathbb{P}^{2n}_{(t_0,\dots,t_{2n})}}(3)|
$$
is the linear system of cubic hypersurfaces of $\mathbb{P}^{2n}_{(t_0,\dots,t_{2n})}$ containing $T_1,T_2,T_4$ and vanishing with multiplicity two on $T_3$, and 
$$
\mathcal{U}_{2n} = |\mathcal{I}_{U_1,U_2}(2)|\subset |\mathcal{O}_{\mathbb{P}^{2n}_{(u_0,\dots,u_{2n})}}(2)|
$$
is the linear system of quadric hypersurfaces containing $U_1,U_2$.

Furthermore, the rational map 
$$
\begin{array}{cccc}
 \alpha_{2n}: & \mathbb{P}^{2n}_{(t_0,\dots,t_{2n})} & \dasharrow & \mathbb{P}^{2n}_{(u_0,\dots,u_{2n})}\\
  & (t_{0},\dots,t_{2n}) & \mapsto & [\alpha_0:\dots:\alpha_{2n}]
\end{array}
$$
is birational and the rational map 
$$
\begin{array}{cccc}
 \beta_{2n}: & \mathbb{P}^{2n}_{(u_0,\dots,u_{2n})} & \dasharrow & \mathbb{P}^{2n}_{(t_0,\dots,t_{2n})}\\
  & (u_{0},\dots,u_{2n}) & \mapsto & [\beta_0:\dots:\beta_{2n}]
\end{array}
$$
is the inverse of $\alpha_{2n}$.
\end{Proposition}
\begin{proof}
The first part of the statement follows from computations analogous to those in the proof of Proposition \ref{linsys4}, and the fact that $\beta_{2n}$ is the inverse of $\alpha_{2n}$ follows from (\ref{alp}) and (\ref{bet}). 
\end{proof}

\begin{Remark}\label{n1}
When $n = 1$ the linear system $\mathcal{T}_{2}$ has the line $\{t_0 = 0\}$ as a base components and it can then be reduced to the linear system of conics 
$$
\left\langle -2t_0^2 - 2t_0t_1 - 2t_1^2, 2t_0t_2 + t_1t_2, t_1t_2\right\rangle
$$
inducing a standard Cremona transformation centered in the two conjugate points $[\frac{-1+\xi}{2}:1:0],[\frac{-1-\xi}{2}:1:0]$ and in $[0:0:1]$. Its inverse is the standard Cremona induced by the linear system
$$
\left\langle -u_0u_1 + u_0u_2,-2u_0u_2,u_1^2 + 3u_2^2\right\rangle
$$
of conics passing through the two conjugate points $[0:\xi:1],[0:-\xi:1]$ and $[1:0:0]$.
\end{Remark}

\begin{thm}\label{rel_cr}
The following diagram of rational maps 
$$
\begin{tikzcd}
{\mathbb{P}^{2n}_{(t_0,\dots,t_{2n})}} \arrow[dd, "\alpha_{2n}", dashed, bend left] \arrow[rrrd, "\overline{\phi}_{2n}", dashed]   &  &    &        \\
                                                                                                                                  &  &  & X^{2n} \\
{\mathbb{P}^{2n}_{(u_0,\dots,u_{2n})}} \arrow[uu, "\beta_{2n}", dashed, bend left] \arrow[rrru, "\overline{\varphi}_{2n}", dashed] &  &    &       
\end{tikzcd}
$$
is commutative.
\end{thm}
\begin{proof}
The map $\overline{\phi}_{2n}$ associates to the points $p^{a_{+}},p^{a_{-}}$ in (\ref{eq+}), (\ref{eq-}) the third intersection point of the line $\left\langle p^{a_{+}},p^{a_{-}}\right\rangle$ and $X^{2n}$, while the map $\overline{\varphi}_{2n}$ associates to the points $x^{a_{+}},x^{a_{-}}$ in Section \ref{res_scal} the third intersection point of the line $\left\langle x^{a_{+}},x^{a_{-}}\right\rangle$ and $X^{2n}$. We will now work out the relation between these two third intersections points. The equalities $x^{a_{+}} = p^{a_{+}}$, $x^{a_{-}} = p^{a_{-}}$ correspond to the following linear system in $u_0,\dots,u_{2n}$:
$$
\left\lbrace\begin{array}{lll}
\frac{u_1-3u_2}{2}+\xi\frac{u_1+u_2}{2} + \frac{(\xi+1)^2t_{2}}{2(\xi+2t_{1}+1)} & = & 0; \\ 
u_1 + \xi u_2 + \frac{(\xi+1)t_{2}}{\xi+2t_{1}+1} & = & 0; \\ 
\frac{u_{2i+3}-3u_{2i+4}}{2}+\xi\frac{u_{2i+3}+u_{2i+4}}{2} + \frac{(\xi+1)(\xi t_{2i+4}+2t_{1}t_{2i+4}-2t_{2i+3}t_{2}+t_{2i+4})}{2(\xi+2t_{1}+1)}  & = & 0; \\ 
u_{2i+3} + \xi u_{2i+4} + \frac{\xi(t_{1}t_{2i+4}-t_{2i+3}t_{2}+t_{2i+4})-t_{1}t_{2i+4}+t_{2i+3}t_{2}+t_{2i+4}}{\xi (t_{1}+1)+t_{1}-1} & = & 0;\\
\frac{u_1-3u_2}{2}-\xi\frac{u_1+u_2}{2} - \frac{(\xi-1)^2t_{2}}{2(\xi-2t_{1}-1)} & = & 0; \\ 
u_1 - \xi u_2 + \frac{(\xi-1)t_{2}}{\xi-2t_{1}-1} & = & 0; \\ 
\frac{u_{2i+3}-3u_{2i+4}}{2}-\xi\frac{u_{2i+3}+u_{2i+4}}{2} - \frac{(\xi-1)(\xi t_{2i+4}-2t_{1}t_{2i+4}+2t_{2i+3}t_{2}-t_{2i+4})}{2(\xi-2t_{1}-1)}  & = & 0; \\ 
u_{2i+3} - \xi u_{2i+4} + \frac{\xi(t_{1}t_{2i+4}-t_{2i+3}t_{2}+t_{2i+4})+t_{1}t_{2i+4}-t_{2i+3}t_{2}-t_{2i+4}}{\xi (t_{1}+1)-t_{1}+1} & = & 0;
\end{array}\right.
$$
for $i = 0,\dots,n-2$. Solving this system with respect to $u_0,\dots,u_{2n}$ we get 
$$
\left\lbrace\begin{array}{lll}
u_1 & = & -\frac{(2+t_1)t_2}{2(t_1^2+t_1+1)};\\
u_{2i} & = & -\frac{t_2t_{2i-1}}{2(t_1^2+t_1+1)};\\
u_{2j+1} & = & -\frac{2(t_1^2t_{2j+2}-t_1t_2t_{2j+1}+t_1t_{2j+2}+t_{2j+2})-t_2t_{2j+1}}{2(t_1^2+t_1+1)};
\end{array}\right.
$$
for $i = 1,\dots,n$, $j = 1,\dots,n-1$, which yield exactly the cubics in (\ref{alp}). Hence $\overline{\varphi}_{2n}\circ \alpha_{2n} = \overline{\phi}_{2n}$ and then the commutativity of the diagram in the statement follows from Proposition \ref{crem}.
\end{proof}

\section{Degenerating degree five del Pezzo and projected K3 surfaces}\label{sec_deg}

In this section we investigate how Fano's rationality construction for cubic $4$-folds containing a del Pezzo surface of degree five behaves in families and in particular how it specializes to our rationality construction for the Fermat cubic $4$-fold in Section \ref{gA}. Most computations in this section can not be carried out by hand, a Magma library containing all the needed scripts can be downloaded at the following link:
\begin{center}
\url{https://github.com/msslxa/Fermat-Cubics}
\end{center}
which contains also several functions to compute the birational parametrizations and their inverses described in the previous sections. 

Let $k(t)$ be the function field in one variable of the base field $k$ and $\mathbb{P}^2_{k(t)}$ the projective plane with homogeneous coordinates $u_0,u_1,u_2$ over $k(t)$. Set
\begin{small}
$$
\begin{array}{lll}
\sigma_0 & = & -u_0^3u_1 - u_1^4 - 3u_0^3u_2 - 6u_1^2u_2^2 - 9u_2^4; \\ 
\sigma_1 & = & u_0^3u_1 + u_1^4 - 3u_0^3u_2 + 6u_1^2u_2^2 + 9u_2^4; \\ 
\sigma_2 & = & -u_0^4 - u_0u_1^3 + 3u_0u_1^2u_2 - 3u_0u_1u_2^2 + 9u_0u_2^3; \\ 
\sigma_3 & = & u_0^4 + u_0u_1^3 + 3u_0u_1^2u_2 + 3u_0u_1u_2^2 + 9u_0u_2^3; \\ 
\sigma_4 & = & t(u_0^4 - u_1^4 - 6u_1^2u_2^2 - 9u_2^4); \\ 
\sigma_5 & = & t(u_0^3u_1 + u_1^4 + 6u_1^2u_2^2 + 9u_2^4); \\ 
\sigma_6 & = & t(u_0^2u_1^2 - u_1^4 + 3u_0^2u_2^2 - 6u_1^2u_2^2 - 9u_2^4); \\ 
\sigma_7 & = & t(u_0u_1^3 + u_1^4 + 3u_0u_1u_2^2 + 6u_1^2u_2^2 + 9u_2^4); \\ 
\sigma_8 & = & tu_0^3u_2; \\ 
\sigma_9 & = & t(u_0u_1^2u_2 + 3u_0u_2^3);
\end{array} 
$$
\end{small}
and consider the map
$$
\begin{array}{cccc}
 \sigma: & \mathbb{P}^{2}_{k(t)} & \dasharrow & \mathbb{P}^{9}_{k(t)}\\
  & [u_0: u_1 :u_2] & \mapsto & [\sigma_0:\dots:\sigma_{9}].
\end{array}
$$
Note that $\sigma_0,\dots,\sigma_3$ are exactly the sections of the linear system $\mathcal{L}_{\overrightarrow{2p}_{1,\pm}^{L_{1,\pm}},p_{2,\pm},p_3}^{4}$ inducing the birational paramatrization of the Fermat cubic surface $X^2$ in Proposition \ref{BiRP2}. The polynomials $\sigma_4,\dots,\sigma_9$ are sections of the linear system $\mathcal{L}_{\overrightarrow{2p}_{1,\pm}^{L_{1,\pm}},p_3}^{4}$ which yields a map $\mathbb{P}^2\dasharrow\mathbb{P}^5$ whose image closure is a smooth del Pezzo surface of degree five.

Now, fix homogeneous coordinates $z_0,\dots,z_9$ on $\mathbb{P}^{9}_{k(t)}$, $x_0,\dots,x_5$ on $\mathbb{P}^{5}_{k(t)}$, consider the linear projection 
$$
\begin{array}{cccc}
 pr: & \mathbb{P}^{9}_{k(t)} & \dasharrow & \mathbb{P}^{5}_{k(t)}\\
  & [z_{0}:\dots :z_{9}] & \mapsto & [z_0:z_2:z_1:z_3:z_4:z_6]
\end{array}
$$
and set $\mathcal{S}_t := \overline{(pr\circ\sigma)(\mathbb{P}^{2}_{k(t)})}\subset \mathbb{P}^{5}_{k(t)}$. Furthermore, set
$$
F_t = t(x_1x_4^2  + x_2x_4x_5 - x_1x_5^2 + x_2x_5^2) + x_0^3 + x_1^3 + x_2^3 + x_3^3 + x_4^3 + x_5^3
$$
and 
\begin{small}
$$
\begin{array}{lll}
\Qoppa_{0,t} & = & \frac{1}{3}t(-x_0^2 + 2x_0x_1 + x_0x_2 - x_0x_3 - x_1x_2 - x_2^2 - x_2x_3) + x_0x_4 - x_1x_5;\\
\Qoppa_{1,t} & = & \frac{1}{3}t(-x_0^2 + x_0x_2 + x_1^2 - x_1x_3 - x_2^2 + x_3^2) + x_0x_5 + x_1x_4 - x_1x_5;\\
\Qoppa_{2,t} & = & \frac{1}{3}t(x_0^2 + x_0x_1 - x_0x_2 + x_0x_3 + x_1x_2 + x_2^2 - 2x_2x_3) + x_2x_4 - x_3x_5;\\
\Qoppa_{3,t} & = & \frac{1}{3}t(x_0^2 - x_0x_2 - x_1^2 + x_1x_3 + x_2^2 - x_3^2) + x_2x_5 + x_3x_4 - x_3x_5;\\
\Qoppa_{4,t} & = & t(x_1x_4 - x_1x_5 + x_2x_5) + x_4^2 - x_4x_5 + x_5^2;
\end{array} 
$$
\end{small}
and consider the smooth cubic $4$-fold
$$
X_t^4 = \{F_t = 0\}\subset \mathbb{P}^{5}_{k(t)}.
$$
Then 
\stepcounter{thm}
\begin{equation}\label{dP5t}
\mathcal{S}_t = \{\Qoppa_{0,t} = \dots = \Qoppa_{4,t} = F_t = 0\}\subset X_t^4 \subset \mathbb{P}^{5}_{k(t)}
\end{equation}
is a smooth degree five del Pezzo surface over $k(t)$. The equation of $X_t^4$ is not necessary to define $\mathcal{S}_t$. Indeed, in order to define $\mathcal{S}_t$ scheme theoretically $\qoppa_{0,t},\dots,\qoppa_{4,t}$ are sufficient. 

However, we want now to consider $\mathcal{S}_t$ as a family $\pi:\mathcal{S}_t\rightarrow \mathbb{A}^1_{k}$ of degree five del Pezzo surfaces over $k$. From this points of view the cubic polynomial $F_0$ is required in order to cut out the fiber $S_0 = \pi^{-1}(0)$. We have that 
$$
S_0 = X^2 \cup H^2_{\pm}
$$ 
where $X^2 = \{x_0^3+x_1^3+x_2^3+x_3^3 = x_4 = x_5 = 0\}\subset \mathbb{P}^{5}_{k}$ is the Fermat cubic surface and 
$$
H^2_{\pm} = 
\left\lbrace
\begin{array}{lll}
x_0^2 - x_0x_1 + x_1^2 = x_2^2 - x_2x_3 + x_3^2 = x_4^2 - x_4x_5 + x_5^2 & = & 0;\\
x_0x_2 - x_0x_3 + x_1x_3 = x_1x_4 + x_0x_5 - x_1x_5 = x_3x_4 + x_2x_5 - x_3x_5 & = & 0;\\
x_1x_2 - x_0x_3 = x_0x_4 - x_1x_5 = x_2x_4 - x_3x_5 & = & 0;
\end{array}\right.
$$
is a degree two surface which geometrically is the union of two skew planes intersecting $X^2$ in two conjugate lines. 

Therefore, $\mathcal{S}_t$ is a family of degree five del Pezzo surfaces specializing to the reducible surface $S_0 = X^2 \cup H^2_{\pm}$, and furthermore $X_t^4$ yields a family of cubic $4$-folds containing $S_t = \pi^{-1}(t)$ and specializing to the Fermat cubic $4$-fold $X_0^4\supset S_0$. Note that the map
$$
\begin{array}{cccc}
 \qoppa_t: & X^4_t & \dasharrow & \mathbb{P}^{4}_{k(t)}\\
  & [x_0: \dots :x_5] & \mapsto & [\Qoppa_{0,t}:\dots:\Qoppa_{4,t}]
\end{array}
$$
specializes exactly to the map $\qoppa_{2n} = \qoppa_0: X^4_0\dasharrow\mathbb{P}^{4}$ in the proof of Proposition \ref{inv1}. We will now analyze the family of inverse maps. The inverse of $\qoppa_t$ has the following form 
$$
\begin{array}{cccc}
 \varphi_t: & \mathbb{P}^{4}_{k(t)} & \dasharrow & X^4_t\\
  & [u_0: \dots :u_4] & \mapsto & [t\varphi_{0,t}+\varphi_{0,0}:\dots:t\varphi_{5,t}+\varphi_{5,0}]
\end{array}
$$
where
\begin{small} 
$$
\begin{array}{lll}
\varphi_{0,0}  & = &  -u_0^4 + 2u_0^3u_1 - 3u_0^2u_1^2 + 2u_0u_1^3 - u_1^4 - u_0u_2^3 - u_1u_2^3 + 3u_0u_2^2u_3 - 3u_0u_2u_3^2 + 2u_0u_3^3 - u_1u_3^3 + 2u_0u_4^3 - u_1u_4^3;\\
\varphi_{1,0}  & = &  u_0^4 - 2u_0^3u_1 + 3u_0^2u_1^2 - 2u_0u_1^3 + u_1^4 + u_0u_2^3 - 2u_1u_2^3 + 3u_1u_2^2u_3 - 3u_1u_2u_3^2 + u_0u_3^3 + u_1u_3^3 + u_0u_4^3 + u_1u_4^3;\\
\varphi_{2,0}  & = &  -u_0^3u_2 + 3u_0^2u_1u_2 - 3u_0u_1^2u_2 + 2u_1^3u_2 - u_2^4 - u_0^3u_3 - u_1^3u_3 + 2u_2^3u_3 - 3u_2^2u_3^2 + 2u_2u_3^3 - u_3^4 + 2u_2u_4^3 - u_3u_4^3;\\
\varphi_{3,0}  & = &  u_0^3u_2 + u_1^3u_2 + u_2^4 - 2u_0^3u_3 + 3u_0^2u_1u_3 - 3u_0u_1^2u_3 + u_1^3u_3 - 2u_2^3u_3 + 3u_2^2u_3^2 - 2u_2u_3^3 + u_3^4 + u_2u_4^3 + u_3u_4^3;\\
\varphi_{4,0}  & = &  -2u_0^3u_4 + 3u_0^2u_1u_4 - 3u_0u_1^2u_4 + u_1^3u_4 - 2u_2^3u_4 + 3u_2^2u_3u_4 - 3u_2u_3^2u_4 + u_3^3u_4 + u_4^4;\\
\varphi_{5,0}  & = &  -u_0^3u_4 - u_1^3u_4 - u_2^3u_4 - u_3^3u_4 - u_4^4;
\end{array} 
$$
\end{small}
and hence $\varphi_t$ specializes to the birational parametrization $\varphi = \varphi_0:\mathbb{P}^{4}\dasharrow X^4_0$ in (\ref{par_rs}). The base locus of $\varphi_t$ yields a family $\mathcal{S}_t'$ of surfaces of degree nine specializing to a surface $S_0'$, where we denote by $S'_t$ the surface of the family $\mathcal{S}_t'$ over $t \in k$. The reduced subscheme of $S_0'$ is the complete intersection 
$$
S'_{0,red} = \{u_0^3 + u_1^3 + u_2^3 + u_3^3 + u_4^3 = u_0^2u_1 - u_0u_1^2 + u_1^3 + u_2^2u_3 - u_2u_3^2 + u_3^3 + u_4^3\}\subset\mathbb{P}^4
$$    
whose singular locus
$$
\Sing(S'_{0,red}) = \{u_0^2 - u_0u_1 + u_1^2 = u_0u_2 - u_0u_3 + u_1u_3 = u_1u_2 - u_0u_3 = u_2^2 - u_2u_3 + u_3^2 = u_4 = 0\}\subset\mathbb{P}^4
$$
is the union of two skew conjugate lines. Note that $S'_{0,red}$ is the scheme $Y^{2n-2}$ in Lemma \ref{Irr} for $n = 2$. The variety $K_0\subset\mathbb{P}^8_{(z_0,\dots,z_8)}$ defined by the following equations
\begin{small}
$$
\left\lbrace\begin{array}{lll}
z_5z_6 - z_4z_7 - z_2z_8 = z_3z_4 - z_1z_6 + z_2z_7 = z_0z_1 + z_3^2 + z_7z_8 = z_2z_4 - z_1z_5 + z_0z_7 = z_0z_5 + z_3z_7 + z_8^2 & = & 0;\\
z_2z_3 - z_4z_8 = z_0z_4 + z_3z_6 = z_0z_2 + z_6z_8 = z_2^2z_4z_6 - z_2^2z_5z_7 - z_4z_5^2z_8 + z_5^3z_8 - z_6z_7^2z_8 + z_7^3z_8 + z_8^4  & = & 0;\\
z_2^2z_5z_6 + z_2^2z_4z_7 - z_2^2z_5z_7 - z_4z_5^2z_8 - z_6z_7^2z_8 = z_4^2 - z_4z_5 + z_5^2 - z_0z_8 = z_1z_4 - z_2z_4 + z_2z_5 - z_0z_6 & = & 0;\\
z_6^2 - z_6z_7 + z_7^2 - z_3z_8 = z_4z_6 - z_4z_7 + z_5z_7 - z_1z_8 = z_3z_5 - z_2z_6 - z_1z_7 + z_2z_7 & = & 0;\\
z_1^2 - z_1z_2 + z_2^2 - z_0z_3 = z_0^2 + z_1z_3 - z_4z_8 + z_5z_8 & = & 0;
\end{array}\right. 
$$
\end{small}
is a smooth $K3$ surface of degree $\deg(K_0) = 14$ and $S'_{0,red}$ is the projection of $K_0$ from the $3$-plane $H_0 = \{z_4 = z_5 = z_6 = z_7 = z_8 = 0\}\subset \mathbb{P}^8$ intersecting $K_0$ in the five points
\stepcounter{thm}
\begin{equation}\label{ptsK3}
[1 : -1 : 0 : 1 : 0 : \dots : 0],\: [0 : a_{\pm} : 1 : 0 : \dots : 0],\: [- 1 + a_{\pm} : a_{\pm} : 0 : 1 : 0 : \dots : 0].
\end{equation}
Note that the scheme
$$
\mathcal{Z}_t = \{u_0^2 - u_0u_1 + u_1^2 = u_0u_2 - u_0u_3 + u_1u_3 = u_1u_2 - u_0u_3 = u_2^2 - u_2u_3 + u_3^2 = u_4 = 0\}\subset\mathbb{P}^4_{k(t)},
$$
defined by the same equations of $\Sing(S'_{0,red})$, is contained in $\mathcal{S}'_t$. Let $\rho_t:\mathbb{P}^4_{k(t)}\dasharrow \mathbb{P}^8_{k(t)}$ be the map induced by the quadrics containing $\mathcal{Z}_t$ and set $\mathcal{K}_t = \overline{\rho_t(\mathcal{S}'_t)}$. Let $\mathcal{H}_t = \{z_4 = z_5 = z_6 = z_7 = z_8 = 0\}\subset \mathbb{P}^8_{k(t)}$ and $\pi_t: \mathcal{K}_t\dasharrow \mathcal{S}'_t$ the projection from $\mathcal{H}_t$. Then $\pi_t$ is the inverse of $\rho_{t|\mathcal{S}'_t}: \mathcal{S}'_t\dasharrow \mathcal{K}_t$. For $t\in k$ general $K_t$ is a smooth surface with arithmetic and geometric genus both equal to one and the following Hodge diamond 
$$
\begin{array}{ccccc}
 &  & 1 &  &  \\ 
 & 0 &  & 0 &  \\ 
1 &  & 22 &  & 1 \\ 
 & 0 &  & 0 &  \\ 
 &  & 1 &  & 
\end{array} 
$$
and $H_t$ intersects $K_t$ in the geometrically reducible curve
$$
C_t = \{z_0^2 + z_0z_3 + z_3^2 = z_1^2 - z_1z_2 + z_2^2 - z_0z_3 = z_4 = \dots = z_8 = 0\}\subset K_t
$$ 
and in the point $p_t = [1 : -1 : 0 : 1 : 0 : \dots : 0]$. The birational map $\pi_t:K_t\dasharrow S'_t$ is not defined in $p_t$ and contracts $C_t$ to the scheme
$$
W_t = \left\lbrace\begin{array}{lll}
u_0^2 + u_0u_3 + u_3^2 = u_0u_1 + u_1u_3 - u_2u_3 + u_3^2 = u_1^2 - u_0u_3 + u_1u_3 - u_2u_3 & = & 0; \\ 
u_0u_2 - u_0u_3 + u_1u_3 = u_1u_2 - u_0u_3 = u_2^2 - u_2u_3 + u_3^2 = u_4 & = & 0;
\end{array}\right.
$$
which is in fact the singular locus of $S'_t$ and geometrically is the union of four distinct points. The inverse map $\rho_{t|S'_t}:S'_t\dasharrow K_t$ contracts the line $\{u_0+u_2 = u_1+u_3 = u_4 = 0\}$ to $p_t$.

Summing-up we constructed a family $\mathcal{S}'_t\subset\mathbb{P}^4_{k(t)}$, whose general member is a surface of degree nine with four singular points, specializing to $S'_0$. However, the family $\mathcal{K}_t$ can not specialize to $K_0$ since for instance $\deg(K_t) = 20$ for $t\in k$ general but $\deg(K_0) = 14$. Indeed the fiber of $\mathcal{K}_t$ over $t = 0$ is the union $K_0\cup E_0$ of two irreducible surfaces where $E_0$ is given by 
$$
\left\lbrace\begin{array}{lll}
z_0^2 + z_0z_3 + z_3^2 = z_4^2 - z_4z_5 + z_5^2 = z_6^2 - z_6z_7 + z_7^2 = z_1^2 - z_1z_2 + z_2^2 - z_0z_3 & = & 0;\\
z_0z_4 + z_1z_6 - z_2z_6 - z_1z_7 = z_1z_4 - z_1z_5 + z_2z_5 - z_3z_7 = z_3z_5 - z_2z_6 - z_1z_7 + z_2z_7 & = & 0;\\
z_2z_4 - z_1z_5 + z_0z_7 = z_3z_4 - z_1z_6 + z_2z_7 = z_0z_5 + z_1z_6 - z_2z_7 = z_0z_5 + z_1z_6 - z_2z_7 & = & 0;\\
z_0z_6 - z_0z_7 - z_3z_7 = z_4z_6 - z_4z_7 + z_5z_7 = z_3z_6 + z_0z_7 = z_5z_6 - z_4z_7 = z_8 & = & 0.
\end{array}\right. 
$$
Note that the projection from $H_0$ contracts $E_0$ to the singular locus of $S'_0$. For a similar construction see \cite[Section 5]{HK07}.

Next, we describe a better behaved family of surfaces coming from a slightly more complicated family of cubic $4$-folds. Let us now consider the family 
$$
X_t^4 = \{G_t = x_0^3 + x_1^3 + x_2^3 + x_3^3 + x_4^3 + x_5^3 + tA + t^2B =  0\}\subset \mathbb{P}^{5}_{k(t)}
$$
where 
$$
A = x_1x_4^2 - x_0x_4x_5 + 3x_2x_4x_5 + x_2x_5^2 - 2x_3x_5^2, \:  B = x_0^2x_5 - x_0x_2x_5 + x_1x_2x_5 + x_2^2x_5 + x_0x_3x_5 - x_2x_3x_5.
$$
Note that $X_t$ contains the del Pezzo surface $\mathcal{S}_t$ in (\ref{dP5t}). As usual we consider the map $\qoppa:X_t^4\dasharrow\mathbb{P}^4_{k(t)}$ given by the quadrics containing $\mathcal{S}_t$ and its inverse $\varphi:\mathbb{P}^4_{k(t)}\dasharrow X_t^4$. As before we will denote by $\mathcal{S}'_t$ the base locus of $\varphi$. 

In order to continue our analysis we will need the following fact.
\begin{Proposition}\label{seg}
Let $S\subset\mathbb{P}^5_k$ be a del Pezzo surface of degree five over a field $k$, $\overline{S} = S\times_{\Spec(k)}\overline{k}\subset \mathbb{P}^5_{\overline{k}}$ its algebraic closure, and $\qoppa:\mathbb{P}^5_{\overline{k}}\dasharrow\mathbb{P}^4_{\overline{k}}$ the map induced by the linear system of quadrics containing $\overline{S}$. Then
\begin{itemize}
\item[(i)] $\overline{S}$ is contained in five Segre $3$-folds $T_i\cong\mathbb{P}^2\times\mathbb{P}^1$;
\item[(ii)] the map $\qoppa$ contracts $T_i$ onto a line $L_i\subset \mathbb{P}^4_{\overline{k}}$.
\end{itemize}
Furthermore, $\qoppa_{|X^4}: X^4\dasharrow \mathbb{P}^4_{\overline{k}}$ is birational. Let $X^4$ be a general cubic $4$-fold containing $\overline{S}$, $\varphi:\mathbb{P}^4_{\overline{k}}\dasharrow X^4$ the inverse of $\qoppa_{|X^4}$, and $S_{\phi}$ the indeterminacy locus of $\phi$. Then $L_i\subset S_{\phi}$.
\end{Proposition}
\begin{proof}
The surface $\overline{S}$ is $\mathbb{P}^2_{\overline{k}}$ blown-up in four points. On $\overline{S}$ there are five pencils of conics coming from the lines through each one of the four points and the conics through all of them. Each conic in each one of these pencils spans a plane and the union of these planes gives rise to a Segre $3$-fold. So we get five Segre $3$-folds $T_i\cong\mathbb{P}^2\times\mathbb{P}^1\supset\overline{S}$. A line in each of the $\mathbb{P}^2$ in $T_i$ intersects $\overline{S}$ in two points. So $\qoppa$ contracts $T_i\cong\mathbb{P}^2\times\mathbb{P}^1$ onto $L\cong \mathbb{P}^1$.

Now, take a general cubic $4$-fold $X^4\supset\overline{S}$. Since $\overline{S}$ is a surface with one apparent double point by the proof of Proposition \ref{oadp} $\qoppa_{|X^4}: X^4\dasharrow \mathbb{P}^4_{\overline{k}}$ is birational. Then $X^4$ intersects a general plane in $T_i$ in a conic contained in $\overline{S}$ plus a line $R\subset X^4$. The line $R$ is contracted to a point by $\qoppa_{|X^4}$, and hence the inverse of $\qoppa_{|X^4}$ is not defined on $L$.       
\end{proof}

We conclude this section with the proof of Theorem \ref{th_spec}.

\begin{proof}[Proof of Theorem \ref{th_spec}]
Proposition \ref{seg} and its proof suggest a way to construct a scheme $\mathcal{Z}_t\subset\mathcal{S}_t$ such that $\overline{\mathcal{Z}}_t\subset\overline{\mathcal{S}}_t$ is the union of two skew lines. The linear system $\mathcal{L}_{\overrightarrow{2p}_{1,\pm}^{L_{1,\pm}},p_3}^{4}$ has base points of multiplicity two at $[0:\xi:1],[0:-\xi:1]$, and hence the lines through these two points are mapped to conics in $\overline{\mathcal{S}}_t$. Take two planes spanned by two general conics in one of these pencils. They are mapped to two points by $\qoppa$ and we denote by $L_{+}$ the line through the these two points. Similarly, considering the other pencil of conics, we construct another line $L_{-}$. By construction these two lines are conjugate and hence $\mathcal{Z}_t = L_{+}\cup L_{-}$ is defined over $k(t)$. Indeed, $\mathcal{Z}_t$ turns out to be defined by the following equations:
$$
\left\lbrace\begin{array}{lll}
t^2u_1u_2 + t(2u_1u_3 - 2u_1u_4 - u_2u_4) - 3u_3u_4 + 3u_4^2 & = & 0;\\
t^2u_2^2 - tu_2u_4 - 4u_3^2 - 4tu_1u_4 + 10u_3u_4 - u_4^2 & = & 0;\\
t(u_2u_3 + u_1u_4 - u_2u_4) + 2u_3^2 - 5u_3u_4 + 2u_4^2 & = & 0;\\
t^2u_1^2 - t(u_1u_3 + u_1u_4) + u_3^2 - u_3u_4 + u_4^2 & = & 0;\\
t(u_0 - u_1 + u_2) + 3u_3 - 3u_4 & = & 0.
\end{array}\right.
$$    
By the proof of Proposition \ref{seg} the fiber of $\qoppa$ over $\mathcal{Z}_t$ is the union of two conjugate Segre $3$-folds containing $\mathcal{S}'_t$. In this case $S'_t$ is a smooth surface of degree nine for $t \in k$ general and $S'_0$ is the base locus of the birational parametrization of the Fermat cubic $4$-fold as in the first part of this section. 

The linear system of quadrics containing $\mathcal{Z}_t$ maps $\mathcal{S}'_t$ to a smooth $K3$ surface $\mathcal{K}_t\subset\mathbb{P}^8$ of degree $\deg(\mathcal{K}_t) = 14$ and $\mathcal{S}'_t$ is the projection of $\mathcal{K}_t$ from a $5$-secant $3$-plane. In this case $K_0$ is exactly the $K3$ surface mapping onto $S'_0$ in the previous construction. 
\end{proof}

\section{Rational points over finite and number fields}\label{ratP}

As an application of the parametrizations in Sections \ref{S4F} and \ref{res_scal} we derive results on the rational points of $X^{2n}$ and of the complete intersection $Y^{2n-2}$ in Lemma \ref{Irr}.

\begin{say}\label{du_map}
In order to do this we need a deeper understanding of the geometry of the birational correspondence 
$$
\begin{tikzcd}
\mathbb{P}^{2n} \arrow[rr, "\overline{\varphi}_{2n}", dashed, bend left] &  & X^{2n} \arrow[ll, "\qoppa_{2n}", dashed, bend left]
\end{tikzcd}
$$
in Sections \ref{res_scal} and \ref{ch2} represented in the following picture 

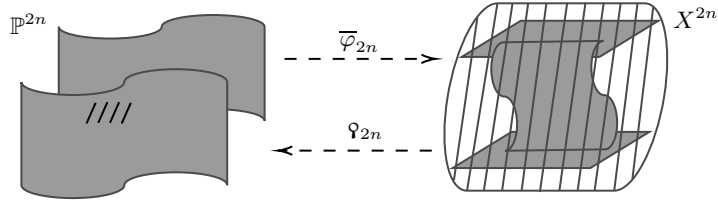
\tikzset{every picture/.style={line width=0.75pt}} 
\begin{center}
\begin{tikzpicture}[x=0.75pt,y=0.75pt,yscale=-0.65,xscale=0.65]

\draw  [color={rgb, 255:red, 74; green, 74; blue, 74 }  ,draw opacity=1 ][fill={rgb, 255:red, 155; green, 155; blue, 155 }  ,fill opacity=0.33 ] (399.3,31) -- (505,31) -- (459.7,59) -- (354,59) -- cycle ;
\draw  [color={rgb, 255:red, 74; green, 74; blue, 74 }  ,draw opacity=1 ][fill={rgb, 255:red, 155; green, 155; blue, 155 }  ,fill opacity=0.28 ] (392.3,117) -- (498,117) -- (452.7,145) -- (347,145) -- cycle ;
\draw  [color={rgb, 255:red, 74; green, 74; blue, 74 }  ,draw opacity=1 ][fill={rgb, 255:red, 155; green, 155; blue, 155 }  ,fill opacity=0.77 ] (462.82,46.58) .. controls (457.51,46.65) and (453.35,56.11) .. (453.51,67.7) .. controls (453.68,79.29) and (458.11,88.63) .. (463.42,88.55) .. controls (468.72,88.48) and (473.15,97.81) .. (473.32,109.4) .. controls (473.48,120.99) and (469.32,130.45) .. (464.02,130.52) -- (387.2,131.62) .. controls (392.5,131.55) and (396.67,122.09) .. (396.5,110.5) .. controls (396.33,98.91) and (391.9,89.57) .. (386.6,89.65) .. controls (381.29,89.73) and (376.86,80.39) .. (376.69,68.8) .. controls (376.53,57.21) and (380.69,47.75) .. (386,47.68) -- cycle ;
\draw [color={rgb, 255:red, 74; green, 74; blue, 74 }  ,draw opacity=1 ]   (476,17.5) -- (456,162.5) ;
\draw [color={rgb, 255:red, 74; green, 74; blue, 74 }  ,draw opacity=1 ]   (463,17) -- (443,162.5) ;
\draw [color={rgb, 255:red, 74; green, 74; blue, 74 }  ,draw opacity=1 ]   (489,17) -- (471.07,162.5) ;
\draw [color={rgb, 255:red, 74; green, 74; blue, 74 }  ,draw opacity=1 ]   (405,17.5) -- (385,162.5) ;
\draw [color={rgb, 255:red, 74; green, 74; blue, 74 }  ,draw opacity=1 ]   (419,17.5) -- (400,162.5) ;
\draw [color={rgb, 255:red, 74; green, 74; blue, 74 }  ,draw opacity=1 ]   (435,17) -- (414,163.5) ;
\draw [color={rgb, 255:red, 74; green, 74; blue, 74 }  ,draw opacity=1 ]   (449,17) -- (429,162.5) ;
\draw  [color={rgb, 255:red, 74; green, 74; blue, 74 }  ,draw opacity=1 ] (381.41,17.5) -- (494.4,17.5) .. controls (509.08,17.5) and (515.76,49.96) .. (509.32,90) .. controls (502.88,130.04) and (485.75,162.5) .. (471.07,162.5) -- (358.08,162.5) .. controls (343.4,162.5) and (336.72,130.04) .. (343.16,90) .. controls (349.6,49.96) and (366.73,17.5) .. (381.41,17.5) -- cycle ;
\draw [color={rgb, 255:red, 74; green, 74; blue, 74 }  ,draw opacity=1 ]   (501,20.5) -- (485,153.5) ;
\draw [color={rgb, 255:red, 74; green, 74; blue, 74 }  ,draw opacity=1 ]   (393,17.5) -- (371,162.5) ;
\draw [color={rgb, 255:red, 74; green, 74; blue, 74 }  ,draw opacity=1 ]   (381.41,17.5) -- (358.08,162.5) ;
\draw [color={rgb, 255:red, 74; green, 74; blue, 74 }  ,draw opacity=1 ]   (366,28.5) -- (348,154.5) ;
\draw  [color={rgb, 255:red, 0; green, 0; blue, 0 }  ,draw opacity=1 ] (66,94.5) -- (100,94.5) .. controls (104.42,94.5) and (108,98.08) .. (108,102.5) .. controls (108,106.92) and (104.42,110.5) .. (100,110.5) -- (66,110.5) .. controls (61.58,110.5) and (58,106.92) .. (58,102.5) .. controls (58,98.08) and (61.58,94.5) .. (66,94.5) -- cycle ;
\draw  [color={rgb, 255:red, 0; green, 0; blue, 0 }  ,draw opacity=1 ][fill={rgb, 255:red, 128; green, 128; blue, 128 }  ,fill opacity=0.77 ] (78.8,64.85) -- (191.5,64.85) -- (143.2,104.35) -- (30.5,104.35) -- cycle ;
\draw  [color={rgb, 255:red, 74; green, 74; blue, 74 }  ,draw opacity=1 ][fill={rgb, 255:red, 155; green, 155; blue, 155 }  ,fill opacity=0.33 ] (44.7,35.65) .. controls (44.7,40.68) and (62.38,44.75) .. (84.2,44.75) .. controls (106.02,44.75) and (123.7,40.68) .. (123.7,35.65) .. controls (123.7,30.62) and (141.38,26.55) .. (163.2,26.55) .. controls (185.02,26.55) and (202.7,30.62) .. (202.7,35.65) -- (202.7,108.45) .. controls (202.7,103.42) and (185.02,99.35) .. (163.2,99.35) .. controls (141.38,99.35) and (123.7,103.42) .. (123.7,108.45) .. controls (123.7,113.48) and (106.02,117.55) .. (84.2,117.55) .. controls (62.38,117.55) and (44.7,113.48) .. (44.7,108.45) -- cycle ;
\draw  [color={rgb, 255:red, 74; green, 74; blue, 74 }  ,draw opacity=1 ][fill={rgb, 255:red, 155; green, 155; blue, 155 }  ,fill opacity=0.33 ] (16,78.47) .. controls (16,83.98) and (33.68,88.45) .. (55.5,88.45) .. controls (77.32,88.45) and (95,83.98) .. (95,78.47) .. controls (95,72.97) and (112.68,68.5) .. (134.5,68.5) .. controls (156.32,68.5) and (174,72.97) .. (174,78.47) -- (174,158.28) .. controls (174,152.77) and (156.32,148.3) .. (134.5,148.3) .. controls (112.68,148.3) and (95,152.77) .. (95,158.28) .. controls (95,163.78) and (77.32,168.25) .. (55.5,168.25) .. controls (33.68,168.25) and (16,163.78) .. (16,158.28) -- cycle ;
\draw  [dash pattern={on 4.5pt off 4.5pt}]  (218,60.5) -- (332,60.5) ;
\draw [shift={(334,60.5)}, rotate = 180] [color={rgb, 255:red, 0; green, 0; blue, 0 }  ][line width=0.75]    (10.93,-3.29) .. controls (6.95,-1.4) and (3.31,-0.3) .. (0,0) .. controls (3.31,0.3) and (6.95,1.4) .. (10.93,3.29)   ;
\draw  [dash pattern={on 4.5pt off 4.5pt}]  (331,131.5) -- (215,131.5) ;
\draw [shift={(213,131.5)}, rotate = 360] [color={rgb, 255:red, 0; green, 0; blue, 0 }  ][line width=0.75]    (10.93,-3.29) .. controls (6.95,-1.4) and (3.31,-0.3) .. (0,0) .. controls (3.31,0.3) and (6.95,1.4) .. (10.93,3.29)   ;
\draw [color={rgb, 255:red, 0; green, 0; blue, 0 }  ,draw opacity=1 ]   (82,94) -- (75,110) ;
\draw [color={rgb, 255:red, 0; green, 0; blue, 0 }  ,draw opacity=1 ]   (100,94.5) -- (93,110.5) ;
\draw [color={rgb, 255:red, 0; green, 0; blue, 0 }  ,draw opacity=1 ]   (91,94) -- (84,110) ;
\draw [color={rgb, 255:red, 0; green, 0; blue, 0 }  ,draw opacity=1 ]   (73,94.5) -- (66,110.5) ;

\draw (5,23.4) node [anchor=north west][inner sep=0.75pt]  [font=\small]  {$\mathbb{P}^{2n}$};
\draw (514,18.4) node [anchor=north west][inner sep=0.75pt]  [font=\small]  {$X^{2n}$};
\draw (258,36.4) node [anchor=north west][inner sep=0.75pt]  [font=\small]  {$\overline{\varphi }_{2n}$};
\draw (264,107.4) node [anchor=north west][inner sep=0.75pt]  [font=\small]  {$\qoppa _{2n}$};
\end{tikzpicture}
\end{center}
which we now discuss. The small ruled area on the left represents the base locus $Y^{2n-2}$ of $\overline{\varphi}_{2n}$ in Lemma \ref{Irr}. The big ruled area on the right, which we will denote by $E^{2n-1}$, is a divisor in $X^{2n}$ which is contract to $Y^{2n-2}$ by $\qoppa_{2n}$. The fibers of $\qoppa_{2n|E^{2n-1}}:E^{2n-1}\dasharrow Y^{2n-2}$ are lines, intersecting each one of the two light grey shaded $n$-planes $H^n_{+},H^n_{-}$ in Proposition \ref{inv1}, and contained in $X^{2n}$.

The light grey shaded areas on the left are two divisors, that we will denote by $D_{+}^{2n-1},D_{-}^{2n-1}$ and which are conjugate over the base field, and get contracted to $H^n_{+},H^n_{-}$ by $\overline{\varphi}_{2n}$. Finally, the dark grey shaded area on the left is the hyperplane $H^{2n-1} = \{u_{2n}=0\}$ which gets contracted to the Fermat cubic $X^{2n-2} = X^{2n}\cap\{x_{2n} = x_{2n+1} = 0\}$ represented by the dark grey shaded area on the right. More specifically, $D^{2n-1} = D_{+}^{2n-1}\cup D_{-}^{2n-1}$ is given by 
$$
D^{2n-1} = \left\lbrace \overline{A}P+\overline{B}Q = 0\right\rbrace
$$
where 
$$
\begin{array}{lll}
P & = & u_0^3 +\sum_{i=1}^n(u_{2i-1}^3+3u_{2i-1}^2u_{2i}+3u_{2i-1}u_{2i}^2+9u_{2i}^3);\\
Q & = & 3u_0^3 +\sum_{i=1}^n(3u_{2i-1}^3-3u_{2i-1}^2u_{2i}+3u_{2i-1}u_{2i}^2-9u_{2i}^3);
\end{array}
$$    
and $\overline{A},\overline{B}$ are the polynomials in Lemma \ref{Irr}.
\end{say}

\begin{proof}[Proof of Theorem \ref{ratpQ}]
The birational parametrization $\overline{\varphi}_{2n}:\mathbb{P}^{2n}\dasharrow X^{2n}\subset\mathbb{P}^{2n+1}$ in Section \ref{res_scal} is given by polynomials of degree four. Hence, $\overline{\varphi}_{2n}$ maps points of $\mathbb{P}^{2n}$ of height at most $B^{\frac{1}{4}}$ to points of $X^{2n}$ of height at most $B$. Let $V\subset X^{2n}$ be the open subset over which $\overline{\varphi}_{2n}$ is finite. The number of points of height at most $B$ of $V$ grows at least as the number of points of height at most $B^{\frac{1}{4}}$ of $V$ which in turn grows as the number of points of height at most $B^{\frac{1}{4}}$ of $\mathbb{P}^{2n}$ minus the number of points of height at most $B^{\frac{1}{4}}$ of a closed subset $Z\subset\mathbb{P}^{2n}$.

Denote by $Z_i$ the irreducible components of $Z$ and let $a_i = \deg(Z_i)$. By \cite[Theorem B]{Pi95} the number of points of height at most $B^{\frac{1}{4}}$ of $Z_i$ grows as $B^{\frac{1}{4}(\dim(Z_i)+\frac{1}{a_i}+\epsilon)}$ with $0< \epsilon\ll 1$. Now, to get the lower bound it is enough to note that by \cite[Theorem 2.1]{Pe02} the number of points of height at most $B^{\frac{1}{4}}$ grows as $B^{\frac{2n+1}{4}}$. When $n = 1$ this lower bound can be improved by using the parametrization in Proposition \ref{BiRP2bis}, which is given by polynomials of degree three, instead of $\overline{\varphi}_{2}$. Moreover, to get the upper bound we can argue in the same way on the map $\qoppa_{2n}:X^{2n}\dasharrow\mathbb{P}^{2n}$ in Proposition \ref{inv1} which is given by polynomials of degree two. Finally, in dimension four it is enough to note that, thanks to the construction in Section \ref{sec_deg}, our argument for the Fermat $4$-fold applies to the general cubic $4$-fold of the family $X_t$ in Theorem \ref{th_spec}.
\end{proof}

Next, we investigate the number of rational points of $X^{2n}$ over finite fields. We begin with the Fermat cubic surface $X^2\subset\mathbb{P}^3$.

\begin{Lemma}\label{modeq}
The polynomial $x^2+3 \in \mathbb{F}_{p^m}[x]$ has a root in $\mathbb{F}_{p^m}$ if and only if either $m$ is even or $p\equiv 1\mod6$ or $p\in\{2,3\}$.
\end{Lemma}
\begin{proof}
Assume that $x^2+3$ has a root in $\mathbb{F}_{p}$. Since $\mathbb{F}_p$ can be embedded in $\mathbb{F}_{p^m}$ as its prime subfield $x^2+3$ has a root in $\mathbb{F}_{p^m}$ as well. Moreover, note that $x^2+3$ has a root in $\mathbb{F}_{p}$ if and only if either $p\equiv 1\mod6$ or $p\in\{2,3\}$.

Now, assume $m$ to be even. If $x^2+3$ does not have a root in $\mathbb{F}_{p}$ then it is irreducible in $\mathbb{F}_{p}[x]$ and $\mathbb{F}_{p^2} \cong \frac{\mathbb{F}_{p}[x]}{(x^2+3)}$. So $x^2+3$ has a root in $\mathbb{F}_{p^2}$. Since $m$ is even $\mathbb{F}_{p^2}\subset \mathbb{F}_{p^m}$ and hence $x^2+3$ has a root in $\mathbb{F}_{p^m}$.

So far we proved that if either $m$ is even or $p\equiv 1\mod6$ or $p\in\{2,3\}$ then $x^2+3$ has a root in $\mathbb{F}_{p^m}$. Hence, to conclude we are left with the case in which $m$ is odd and $p\equiv 5\mod6$. Let $a\in\mathbb{F}_{p^m}$ be a root of $x^2+3$. Then
$$
a^{p^m-1} = (-3)^{\frac{p^m-1}{2}} = \left((-3)^{\frac{p-1}{2}}\right)^{\sum_{i=0}^{m-1}p^i}.
$$
On the other hand $a^{p^m-1} = 1$. Since $p\equiv 5\mod6$ we have that $(-3)^{\frac{p-1}{2}} = -1$, and since $m$ is odd $\sum_{i=0}^{m-1}p^i$ is also odd. Hence $a^{p^m-1} = (-1)^{\sum_{i=0}^{m-1}p^i} = -1$, a contradiction. 
\end{proof} 

\begin{Corollary}\label{ptsff}
Let $X^2\subset\mathbb{P}^3$ be the Fermat cubic surfaces over a finite field $k = \mathbb{F}_{p^m}$. Then the number of points of $X^2$ is given by 
$$
|X^2(\mathbb{F}_q)| = 
\left\lbrace
\begin{array}{ll}
q^2+7q+1 & \text{if } m \text{ is even or }  p\equiv 1\mod6;\\ 
q^2+q+1 & \text{otherwise}.
\end{array}\right.
$$
\end{Corollary}
\begin{proof}
We will take advantage of the parametrization in Proposition \ref{BiRP2bis} which yields an isomorphism between the blow-up $\widetilde{\mathbb{P}}^2_{(v_0,v_1,v_2)}$ of $\mathbb{P}^2_{(v_0,v_1,v_2)}$ at the three pairs of conjugate points $q_{1,\pm},q_{2,\pm},q_{3,\pm}$ and $X^2$. Denote by $E_{i,\pm}$ the exceptional divisor over $q_{i,\pm}$ for $i = 1,2,3$. Recall that $k(\xi)$ is a quadratic extension of $k$ with $\xi^2 = -3$.

First assume that $p\notin\{2,3\}$. If $m$ is even or $p\equiv 1\mod6$ then Lemma \ref{modeq} yields that $-3$ is a square in $\mathbb{F}_{p^m}$. So $\xi \in \mathbb{F}_{p^m}$ and the blown-up points $q_{i,\pm}$ are defined over the base field. Hence, the six exceptional divisors $E_{i,\pm}\cong \mathbb{P}^1$ are defined over the base field as well. Therefore
$$
|X^2(\mathbb{F}_q)|  = (|\mathbb{P}^2(\mathbb{F}_q)|-6) + 6|\mathbb{P}^1(\mathbb{F}_q)| = (q^2+q+1-6)+6(q+1) = q^2+7q+1.
$$
If $m$ is odd and $p\not\equiv 1\mod6$ then by Lemma \ref{modeq} we have that $-3$ is not a square in $\mathbb{F}_{p^m}$ so that the blown-up points are not defined over $\mathbb{F}_{p^m}$. In this case the birational parametrization $\chi:\mathbb{P}^2\dasharrow X^2$ in Proposition \ref{BiRP2bis} yields a bijection between $X^2(\mathbb{F}_q)$ and $\mathbb{P}^2(\mathbb{F}_q)$ so that 
$$
|X^2(\mathbb{F}_q)| = |\mathbb{P}^2(\mathbb{F}_q)| = q^2+q+1.
$$
If $p = 3$ then the reduced subscheme of $X^2$ is the plane $\{x_0+x_1+x_2+x_3 = 0\}$ and hence 
$$|X(\mathbb{F}_q)| = |\mathbb{P}^2(\mathbb{F}_q)| = q^2+q+1 = 3^{2m}+3^m+1.$$
Finally, consider the case $p = 2$. If $m$ is odd then the birational parametrization $\beta$ in Proposition \ref{Fer_Char2} is not defined in three pairs of conjugate points and hence $|X^2(\mathbb{F}_{q})| = q^2+q+1$. If $m$ is even then the base scheme of $\beta$ consists of six points defined over the base field and hence $|X^2(\mathbb{F}_{q})| = q^2+7q+1$.
\end{proof}

\begin{Remark}\label{AO}
If $q\equiv 2\mod3$ and $X^N\subset\mathbb{P}^{N+1}$ is a cubic hypersurface of the form
$$
X^N = \{f(x_0,\dots,x_N)+x_{N+1}^3 = 0\}\subset\mathbb{P}^{N+1}
$$
the projection $X^N\rightarrow\mathbb{P}^N_{(x_0,\dots,x_N)}$ yields a bijection between $X^N(\mathbb{F}_q)$ and $\mathbb{P}^N(\mathbb{F}_q)$ \cite[Observation 1.7.2]{Ked12}, \cite[Remark 4.10]{DLR17}. So $\sharp X^N(\mathbb{F}_q) = \frac{q^{N+1}-1}{q-1}$. In particular, for the Fermat cubic $X^{N}\subset\mathbb{P}^{N+1}$ we have that $\sharp X^{2n}(\mathbb{F}_q) = \frac{q^{2n+1}-1}{q-1}$.

By the work of A. Weil when $q\equiv 1\mod3$ the number of points of the Fermat cubic hypersurface $X^{N}\subset\mathbb{P}^{N+1}$ can be computed as the sums of $\frac{q^{N+1}-1}{q-1}$ and an additional term involving the Jacobi sum of certain non trivial characters of $\mathbb{F}_q$ \cite{Wei49}.

Finally, when $q\equiv 0\mod3$ then $x_0^3+\dots + x_N^3 = (x_0+\dots +x_N)^3$ and hence $\sharp X^N(\mathbb{F}_q) = \frac{q^{N+1}-1}{q-1}$.
\end{Remark}

\begin{Proposition}\label{FF_Y}
Let $k = \mathbb{F}_q$. If $q\equiv 2\mod3$ and $n\geq 2$ then the number of points of the complete intersection $Y^{2n-2}\subset\mathbb{P}^{2n}$ is given by 
$$
\sharp Y^{2n-2}(k) = \frac{q^{2n-1}-1}{q-1}
$$
that is $\sharp Y^{2n-2}(k) = \sharp \mathbb{P}^{2n-2}(k)$. 
\end{Proposition}
\begin{proof}
We will follow the notation of \ref{du_map}. First, consider the intersection $W^{2n-3} = Y^{2n-2}\cap H^{2n-1}$. Assume that $n\geq 3$. Then the projection $\pi_p$ from the point $p= [0:\dots:0:1]\in H^{2n-1}$ maps $W^{2n-3}$ to the cubic hypersurface 
$$
\overline{W}^{2n-3} = \left\lbrace\sum_{i=1}^{n-1}v_{2i-1}^2v_{2i}+3v_{2i}^3 = 0\right\rbrace\subset\mathbb{P}^{2n-2}_{(v_0,\dots,v_{2n-2})}.
$$
Note that $p\notin W^{2n-3}$. Fix a point $q = [q_0:\dots:q_{2n-2}]\in\overline{W}^{2n-3}$ and let $q_i$ be the first non zero homogeneous coordinate of $q$. Then the fiber of $\pi_{p|W^{2n-3}}:W^{2n-3}\rightarrow \overline{W}^{2n-3}$ over $q$ is defined, in the line $\left\langle p,q\right\rangle$ by $\{p_i^3u_{2n-2-i}^3+u_{2n-1}^3 = 0\}$. Hence, such fiber consists of two conjugate points and a point defined over the base field, and Remark \ref{AO} yields that 
\stepcounter{thm}
\begin{equation}\label{eqW}
\sharp W^{2n-3}(k) = \sharp\overline{W}^{2n-3}(k) =  \frac{q^{2n-2}-1}{q-1}.
\end{equation}
Since $D^{2n-1}_{+},D^{2n-1}_{-}$ are not defined over $k$ all their points must lies in $Y^{2n-2}$, and since $H^{n}_{+},H^{n}_{-}$ are also not defined over $k$ their union $H^{n}_{+}\cup H^{n}_{-}$ does not have points. Moreover, the maps $\overline{\varphi}_{2n},\qoppa_{2n}$ define a biregular correspondence outside of these loci. Therefore,
$$
\sharp \mathbb{P}^{2n}(k) -\sharp Y^{2n-2}(k) -\sharp H^{2n-1}(k) + \sharp W^{2n-3}(k) = \sharp X^{2n}(k)-\sharp X^{2n-2}(k)-(\sharp Y^{2n-2}(k)-\sharp W^{2n-3}(k))\cdot \sharp \mathbb{P}^{1}(k)
$$
where we took into account that the fibers of $E^{2n-1}$ over $Y^{2n-2}\setminus Z^{n-1}$ are lines, and that the singular locus $Z^{n-1}$ of $Y^{2n-2}$ in Lemma \ref{lsing} does not have points. 

Now, Remark \ref{AO} yields that $\sharp X^{2n}(k) = \sharp \mathbb{P}^{2n}(k)$ and $\sharp X^{2n-2}(k) = \sharp \mathbb{P}^{2n-2}(k)$. Therefore, taking into account (\ref{eqW}) we get that 
$$
\sharp Y^{2n-2}(k)\left(\frac{q^2-1}{q-1}-1 \right) = \frac{q^{2n}-q^{2n-1}-q^{2n-2}+1}{q-1}+(q+1)\frac{q^{2n-2}-1}{q-1}
$$
and hence 
$$
\sharp Y^{2n-2}(k)q = \frac{q(q^{2n-1}-1)}{q-1}
$$
concluding the proof in the case $n\geq 3$. When $n = 2$ the intersection $\overline{W}^{2n-3}$ has four irreducible components: three of these are geometrically union of two skew lines and therefore with no points, the fourth one is the cubic $\{u_0^3+u_1^3+u_2^3 = u_3 = u_4 = 0\}$ which by Remark \ref{AO} has $\sharp\mathbb{P}^1(k)$ points. Hence (\ref{eqW}) holds also for the case $n = 2$ and the proof goes through as for the case $n\geq 3$.
\end{proof}

\begin{Remark}
In \cite{Hoo91} C. Hooley, generalizing a result of P. Deligne for smooth complete intersections \cite{Del74}, proved that the number of points of an $n$-dimensional complete intersection over $\mathbb{F}_q$ whose singular locus has dimension $d$ is given by 
$$
\frac{q^{n+1}-1}{q-1}+ O(q^{\frac{n+d+1}{2}}).
$$
Proposition \ref{FF_Y} provides a class of $(2n-2)$-dimensional singular varieties having the same number of points of $\mathbb{P}^{2n-2}$.
\end{Remark}

\begin{Corollary}\label{K3FF}
Let $k = \mathbb{F}_q$. Assume that $q\equiv 2\mod3$, $n\geq 2$ and consider the $K3$ surface $K_0\subset\mathbb{P}^8$ in Section \ref{sec_deg} over $k$. Then 
$$
\sharp K_0(k) = \frac{q^3-1}{q-1}-\frac{q^2-1}{q-1}+1.
$$
\end{Corollary}
\begin{proof}
Recall that $Y^2\subset\mathbb{P}^4_{(u_0,\dots,u_4)}$ is the projection of $K_0\subset\mathbb{P}^8$ from the $3$-plane $H_0$. Note that by (\ref{ptsK3}) the intersection $K_0\cap H_0$ has just one point $p$ defied over $k$. The inverse of the projection is given by the quadrics containing the singular locus $\Sing(Y^2)$ of $Y^2$ which consists of two skew conjugate lines and does not have any point. Such inverse contracts just the line $L_p := \{u_0+u_2=u_1+u_3 = u_4=0\} \subset Y^2$ to the point $p$.

The projection maps a curve of degree four onto $\Sing(Y^2)$ and contracts five pairs of conjugate lines to five pairs of conjugate points on $\Sing(Y^2)$. However, since $\Sing(Y^2)$ does not have points these curves also do not have points. Hence 
$$
\sharp K_0(k) = \sharp Y^2(k) - \sharp L_p +\sharp \{p\}
$$
and to conclude the proof it is enough to apply Proposition \ref{FF_Y}.
\end{proof}

\bibliographystyle{amsalpha}
\bibliography{Biblio}

\providecommand{\bysame}{\leavevmode\hbox to3em{\hrulefill}\thinspace}
\providecommand{\MR}{\relax\ifhmode\unskip\space\fi MR }
\providecommand{\MRhref}[2]{%
  \href{http://www.ams.org/mathscinet-getitem?mr=#1}{#2}
}
\providecommand{\href}[2]{#2}
\begin{thebibliography}{BCTSSD85}

\bibitem[AO18]{AO18}
H.~Ahmadinezhad and T.~Okada, \emph{Stable rationality of higher dimensional
  conic bundles}, \'{E}pijournal G\'{e}om. Alg\'{e}brique \textbf{2} (2018),
  Art. 5, 13. \MR{3816900}

\bibitem[AT14]{AT14}
N.~Addington and R.~Thomas, \emph{Hodge theory and derived categories of cubic
  fourfolds}, Duke Math. J. \textbf{163} (2014), no.~10, 1885--1927.
  \MR{3229044}

\bibitem[BCTSSD85]{BCSS85}
A.~Beauville, J.~L. Colliot-Th\'{e}l\`ene, J.~J. Sansuc, and
  P.~Swinnerton-Dyer, \emph{Vari\'{e}t\'{e}s stablement rationnelles non
  rationnelles}, Ann. of Math. (2) \textbf{121} (1985), no.~2, 283--318.
  \MR{786350}

\bibitem[BG06]{BoG06}
E.~Bombieri and W.~Gubler, \emph{Heights in {D}iophantine geometry}, New
  Mathematical Monographs, vol.~4, Cambridge University Press, Cambridge, 2006.
  \MR{2216774}

\bibitem[BRS19]{BRS19}
M.~Bolognesi, F.~Russo, and G.~Staglian\`o, \emph{Some loci of rational cubic
  fourfolds}, Math. Ann. \textbf{373} (2019), no.~1-2, 165--190. \MR{3968870}

\bibitem[BvB18]{BG18}
C.~B\"{o}hning and H.~C.~Graf von Bothmer, \emph{On stable rationality of some
  conic bundles and moduli spaces of {P}rym curves}, Comment. Math. Helv.
  \textbf{93} (2018), no.~1, 133--155. \MR{3777127}

\bibitem[CG72]{CG72}
C.~H. Clemens and P.~A. Griffiths, \emph{The intermediate {J}acobian of the
  cubic threefold}, Ann. of Math. (2) \textbf{95} (1972), 281--356. \MR{302652}

\bibitem[CTP16]{CP16}
J.~L. Colliot-Th\'{e}l\`ene and A.~Pirutka, \emph{Hypersurfaces quartiques de
  dimension 3: non-rationalit\'{e} stable}, Ann. Sci. \'{E}c. Norm. Sup\'{e}r.
  (4) \textbf{49} (2016), no.~2, 371--397. \MR{3481353}

\bibitem[CTSSD87]{CSS87}
J.~L. Colliot-Th\'{e}l\`ene, J.~J. Sansuc, and P.~Swinnerton-Dyer,
  \emph{Intersections of two quadrics and {C}h\^{a}telet surfaces. {I}}, J.
  Reine Angew. Math. \textbf{373} (1987), 37--107. \MR{870307}

\bibitem[Del74]{Del74}
P.~Deligne, \emph{La conjecture de {W}eil. {I}}, Inst. Hautes \'{E}tudes Sci.
  Publ. Math. (1974), no.~43, 273--307. \MR{340258}

\bibitem[DLR17]{DLR17}
O.~Debarre, A.~Laface, and X.~Roulleau, \emph{Lines on cubic hypersurfaces over
  finite fields}, Geometry over nonclosed fields, Simons Symp., Springer, Cham,
  2017, pp.~19--51. \MR{3644249}

\bibitem[Fan43]{Fan43}
G.~Fano, \emph{Sulle forme cubiche dello spazio a cinque dimensioni contenenti
  rigate razionali del {$4^\circ$} ordine}, Comment. Math. Helv. \textbf{15}
  (1943), 71--80. \MR{10433}

\bibitem[FLS18]{FLS18}
C.~Frei, D.~Loughran, and E.~Sofos, \emph{Rational points of bounded height on
  general conic bundle surfaces}, Proc. Lond. Math. Soc. (3) \textbf{117}
  (2018), no.~2, 407--440. \MR{3851328}

\bibitem[FMT89]{FMT89}
J.~Franke, Y.~I. Manin, and Y.~Tschinkel, \emph{Rational points of bounded
  height on {F}ano varieties}, Invent. Math. \textbf{95} (1989), no.~2,
  421--435. \MR{974910}

\bibitem[Has99]{Has99}
B.~Hassett, \emph{Some rational cubic fourfolds}, J. Algebraic Geom. \textbf{8}
  (1999), no.~1, 103--114. \MR{1658216}

\bibitem[Has00]{Has00}
\bysame, \emph{Special cubic fourfolds}, Compositio Math. \textbf{120} (2000),
  no.~1, 1--23. \MR{1738215}

\bibitem[HK07]{HK07}
K.~Hulek and R.~Kloosterman, \emph{The {$L$}-series of a cubic fourfold},
  Manuscripta Math. \textbf{124} (2007), no.~3, 391--407. \MR{2350552}

\bibitem[HKT16]{HKT16}
B.~Hassett, A.~Kresch, and Y.~Tschinkel, \emph{Stable rationality and conic
  bundles}, Math. Ann. \textbf{365} (2016), no.~3-4, 1201--1217. \MR{3521088}

\bibitem[Hoo91]{Hoo91}
C.~Hooley, \emph{On the number of points on a complete intersection over a
  finite field}, J. Number Theory \textbf{38} (1991), no.~3, 338--358, With an
  appendix by Nicholas M. Katz. \MR{1114483}

\bibitem[HPT18]{HPT18}
B.~Hassett, A.~Pirutka, and Y.~Tschinkel, \emph{Stable rationality of quadric
  surface bundles over surfaces}, Acta Math. \textbf{220} (2018), no.~2,
  341--365. \MR{3849287}

\bibitem[HPT19]{HPT19}
\bysame, \emph{A very general quartic double fourfold is not stably rational},
  Algebr. Geom. \textbf{6} (2019), no.~1, 64--75. \MR{3904799}

\bibitem[HW08]{HW08}
G.~H. Hardy and E.~M. Wright, \emph{An introduction to the theory of numbers},
  sixth ed., Oxford University Press, Oxford, 2008, Revised by D. R.
  Heath-Brown and J. H. Silverman, With a foreword by Andrew Wiles.
  \MR{2445243}

\bibitem[Ked12]{Ked12}
K.~S. Kedlaya, \emph{Effective {$p$}-adic cohomology for cyclic cubic
  threefolds}, Computational algebraic and analytic geometry, Contemp. Math.,
  vol. 572, Amer. Math. Soc., Providence, RI, 2012, pp.~127--171. \MR{2953828}

\bibitem[Kol95]{Ko95}
J.~Koll\'{a}r, \emph{Nonrational hypersurfaces}, J. Amer. Math. Soc. \textbf{8}
  (1995), no.~1, 241--249. \MR{1273416}

\bibitem[Kol02]{Kol02}
J\'{a}nos Koll\'{a}r, \emph{Unirationality of cubic hypersurfaces}, J. Inst.
  Math. Jussieu \textbf{1} (2002), no.~3, 467--476. \MR{1956057}

\bibitem[KT19]{KT19}
M.~Kontsevich and Y.~Tschinkel, \emph{Specialization of birational types},
  Invent. Math. \textbf{217} (2019), no.~2, 415--432. \MR{3987175}

\bibitem[Kuz10]{Kuz10}
A.~Kuznetsov, \emph{Derived categories of cubic fourfolds}, Cohomological and
  geometric approaches to rationality problems, Progr. Math., vol. 282,
  Birkh\"{a}user Boston, Boston, MA, 2010, pp.~219--243. \MR{2605171}

\bibitem[Mor40]{Mor40}
U.~Morin, \emph{Sulla razionalit\`a dell'ipersuperficie cubica generale dello
  spazio lineare {$S_5$}}, Rend. Sem. Mat. Univ. Padova \textbf{11} (1940),
  108--112. \MR{20273}

\bibitem[Pey95]{Pey95}
E.~Peyre, \emph{Hauteurs et mesures de {T}amagawa sur les vari\'{e}t\'{e}s de
  {F}ano}, Duke Math. J. \textbf{79} (1995), no.~1, 101--218. \MR{1340296}

\bibitem[Pey02]{Pe02}
\bysame, \emph{Points de hauteur born\'{e}e et g\'{e}om\'{e}trie des
  vari\'{e}t\'{e}s (d'apr\`es {Y}. {M}anin et al.)}, no. 282, 2002,
  S\'{e}minaire Bourbaki, Vol. 2000/2001, pp.~Exp. No. 891, ix, 323--344.
  \MR{1975184}

\bibitem[Pil95]{Pi95}
J.~Pila, \emph{Density of integral and rational points on varieties}, no. 228,
  1995, Columbia University Number Theory Seminar (New York, 1992), pp.~4,
  183--187. \MR{1330933}

\bibitem[RS19]{RS19}
F.~Russo and G.~Staglian\`o, \emph{Congruences of 5-secant conics and the
  rationality of some admissible cubic fourfolds}, Duke Math. J. \textbf{168}
  (2019), no.~5, 849--865. \MR{3934590}

\bibitem[RS23]{RS23}
\bysame, \emph{Trisecant flops, their associated {K}3 surfaces and the
  rationality of some cubic fourfolds}, J. Eur. Math. Soc. (JEMS) \textbf{25}
  (2023), no.~6, 2435--2482. \MR{4592873}

\bibitem[Sch19a]{Sc19a}
S.~Schreieder, \emph{On the rationality problem for quadric bundles}, Duke
  Math. J. \textbf{168} (2019), no.~2, 187--223. \MR{3909896}

\bibitem[Sch19b]{Sc19b}
\bysame, \emph{Stably irrational hypersurfaces of small slopes}, J. Amer. Math.
  Soc. \textbf{32} (2019), no.~4, 1171--1199. \MR{4013741}

\bibitem[Seg43]{Seg43}
B.~Segre, \emph{A note on arithmetical properties of cubic surfaces}, J. London
  Math. Soc \textbf{18} (1943), 24--31. \MR{0009471}

\bibitem[Tot16]{To16}
B.~Totaro, \emph{Hypersurfaces that are not stably rational}, J. Amer. Math.
  Soc. \textbf{29} (2016), no.~3, 883--891. \MR{3486175}

\bibitem[Voi86]{Voi86}
C.~Voisin, \emph{Th\'{e}or\`eme de {T}orelli pour les cubiques de {${\bf
  P}^5$}}, Invent. Math. \textbf{86} (1986), no.~3, 577--601. \MR{860684}

\bibitem[Voi15]{Vo15}
\bysame, \emph{Unirational threefolds with no universal codimension {$2$}
  cycle}, Invent. Math. \textbf{201} (2015), no.~1, 207--237. \MR{3359052}

\bibitem[Wei49]{Wei49}
A.~Weil, \emph{Numbers of solutions of equations in finite fields}, Bull. Amer.
  Math. Soc. \textbf{55} (1949), 497--508. \MR{29393}

\bibitem[YY23]{YY23}
S.~Yang and X.~Yu, \emph{On lattice polarizable cubic fourfolds}, Res. Math.
  Sci. \textbf{10} (2023), no.~1, Paper No. 2, 39. \MR{4519216}

\end{thebibliography}

\end{document}